\documentclass[12pt,twoside]{amsart}

\usepackage{geometry}
\usepackage{amssymb, amsmath, amsthm, amscd}
\usepackage{enumerate}
\usepackage{mathrsfs}
\usepackage{bm}

\usepackage{graphicx}
\usepackage{hhline}
\usepackage[all]{xy}

\usepackage{xcolor}

\usepackage{url}

\usepackage{iftex}

\ifPDFTeX

  \usepackage[pdfencoding=auto, unicode]{hyperref}
\else
  \usepackage[dvipdfmx, unicode]{hyperref}
\fi

\hypersetup{
    colorlinks=false,
    linkcolor=blue,
    citecolor=blue,
    urlcolor=blue,
    bookmarksnumbered=true,
    pdfstartview=FitH
}

\title{Notes on acceptable bundles II} 
\author{Osamu Fujino, Taro Fujisawa, and Takashi Ono}
\date{2026/4/1, version 0.07}
\address{Department of 
Mathematics, Graduate School of Science, 
Kyoto University, Kyoto 606-8502, Japan}
\email{fujino@math.kyoto-u.ac.jp}
\address{Department of Mathematics and Data Science,
Center for Liberal Arts and Sciences,
Tokyo Denki University, Tokyo, Japan}
\email{fujisawa@mail.dendai.ac.jp}
\address{Research Institute for Mathematical Sciences, Kyoto University, Kyoto 606-8502, Japan}
\email{takashio@kurims.kyoto-u.ac.jp}
\DeclareMathOperator{\Id}{Id}
\DeclareMathOperator{\ord}{ord}
\DeclareMathOperator{\rank}{rank}
\DeclareMathOperator{\Nak}{Nak}
\DeclareMathOperator{\dvol}{dvol}
\DeclareMathOperator{\End}{End}
\DeclareMathOperator{\Hom}{Hom}
\DeclareMathOperator{\Gr}{Gr}
\DeclareMathOperator{\re}{Re}
\DeclareMathOperator{\im}{Im}
\DeclareMathOperator{\Ker}{Ker}
\DeclareMathOperator{\Par}{\mathcal{P}\!\it{ar}}

\newtheorem{thm}{Theorem}[section]
\newtheorem{lem}[thm]{Lemma}
\newtheorem{cor}[thm]{Corollary}
\newtheorem{prop}[thm]{Proposition}

\newtheorem{claim}[thm]{Claim}

\theoremstyle{definition}
\newtheorem{defn}[thm]{Definition}
\newtheorem{rem}[thm]{Remark}

\newtheorem*{ack}{Acknowledgments}  
\newtheorem{step}{Step}
\newtheorem{say}[thm]{}

\makeatletter

\@addtoreset{equation}{section}
\makeatother

\begin{document}

\begin{abstract} 
The notion of acceptable bundles plays a fundamental role in the
Simpson--Mochizuki theory. We study acceptable bundles on a partially punctured polydisk in detail. While this article is primarily expository, it also presents new arguments that differ from those of Mochizuki. 
\end{abstract}

\maketitle 

\tableofcontents 

\section{Introduction}\label{x-sec1}

This paper is a continuation of \cite{fujino-fujisawa-ono}, where we
gave a detailed study of acceptable bundles on a punctured disk.
Here we extend the theory to higher-dimensional settings.
More precisely, we investigate acceptable bundles on a partially
punctured polydisk.

The notion of acceptable bundles plays a fundamental role in the
Simpson--Mochizuki theory; see, for example, 
\cite{simpson1}, \cite{simpson2}, \cite{mochizuki1}, 
\cite{mochizuki2}, \cite{mochizuki3}, \cite{mochizuki4}, 
and \cite{mochizuki5}.
It has also found important applications in the study of
higher-dimensional complex varieties.
For related developments, we refer the reader to, for example,
\cite{deng}, \cite{deng-cadorel}, \cite{deng-hao}, and \cite{kim}.

Throughout this paper, we freely use the results established in
\cite{fujino-fujisawa-ono}.
In particular, the theory of acceptable bundles over a punctured disk
developed there plays a decisive role in the present work.
As in \cite{fujino-fujisawa-ono}, one of the main purposes of this paper
is to make Mochizuki's theory of acceptable bundles
\cite[Chapter~21, Acceptable Bundles]{mochizuki4} more accessible to a
broader audience.

Although we use an $L^2$ extension theorem of Ohsawa--Takegoshi type as
a black box, we aim to present the theory of acceptable bundles in a
form that is as self-contained as possible and accessible at the level
of \cite{demailly1}.

\medskip

Let $E$ be a holomorphic vector bundle on the partially punctured
polydisk $(\Delta^{*})^{l} \times \Delta^{\,n-l}$, and let $h$ be a smooth
Hermitian metric on $E$.
We denote by $\omega_{P}$ the Poincar\'e metric on
$(\Delta^{*})^{l} \times \Delta^{\,n-l}$, defined by
\[
  \omega_{P}
  := \sum_{j=1}^{l}
       \frac{\sqrt{-1}\, dz_{j} \wedge d\overline{z}_{j}}
            {|z_{j}|^{2}\bigl(-\log |z_{j}|^{2}\bigr)^{2}}
     + \sum_{k=l+1}^{n}
       \frac{\sqrt{-1}\, dz_{k} \wedge d\overline{z}_{k}}
            {(1 - |z_{k}|^{2})^{2}} .
\]

We say that $(E,h)$ is \emph{acceptable} if the curvature
$\sqrt{-1}\Theta_{h}(E)$, which is a smooth $\Hom(E,E)$-valued $(1,1)$-form,
is bounded with respect to the metric induced by $h$ and $\omega_{P}$.

Let $\bm a = (a_{1},\dots,a_{l}) \in \mathbb{R}^{l}$.
We define an $\mathcal O_{\Delta^n}$-module ${}_{\bm a}E$ as follows.
For any open set $U \subset \Delta^{n}$, set
\[
  \Gamma(U, {}_{\bm a}E)
  :=
  \Bigl\{
     s \in \Gamma\!\bigl(U \cap ((\Delta^{*})^{l}\times \Delta^{\,n-l}), E\bigr)
     \,\Bigm|\,
     |s|_{h}
       = O\!\left(
           \frac{1}{\prod_{j=1}^{l} |z_{j}|^{\,a_{j}+\varepsilon}}
         \right)
     \text{ for every }\varepsilon>0
  \Bigr\}.
\]
We note that we sometimes write $\mathcal P_{\bm a} E$ instead of
${}_{\bm a} E$. Moreover, ${}_{\bm 0} E$ is usually denoted by 
${}^\diamond\!E$, where $\bm 0=(0, \ldots, 0)\in \mathbb R^l$.  

\medskip

One of the main results of this paper is the following.

\begin{thm}[{Prolongation by increasing orders, 
cf.~\cite[Theorem~21.3.1]{mochizuki4}}]
\label{x-thm1.1}
Let $(E, h)$ be an acceptable vector bundle on a partially punctured
polydisk $(\Delta^*)^l\times \Delta^{n-l}$.
Then ${}_{\bm a} E$ is a locally free sheaf on $\Delta^n$ for any
$\bm a \in \mathbb R^l$.
Moreover, the family $\left( {}_{\bm a} E \mid \bm a\in \mathbb R^l 
\right)$
naturally forms a filtered bundle.
\end{thm}

We make a brief remark on the assumptions in Theorem~\ref{x-thm1.1}.

\begin{rem}\label{x-rem1.2}
In \cite[Theorem~21.3.1]{mochizuki4}, Mochizuki assumes for simplicity
that $(\det E, \det h)$ is flat.
\end{rem}
More precisely, we prove the following statement.

\begin{thm}\label{x-thm1.3}
Let $(E, h)$ be an acceptable vector bundle on a partially punctured
polydisk $(\Delta^*)^l \times \Delta^{n-l}$ with $\rank E = r$.
Fix $\bm a \in \mathbb R^l$.
Then there exists a sufficiently small open neighborhood
$U$ of the origin $(0,\ldots,0)$ in $\Delta^n$ such that
there are a local frame $\bm v = \{v_1,\ldots,v_r\}$ of ${}_{\bm a} E$
and vectors $\bm a(v_j) \in \mathbb R^l$ $(j=1,\ldots,r)$
with the following property:
for any $\bm b \in \mathbb R^l$, we have
\[
{}_{\bm b} E
=
\bigoplus_{j=1}^r
\mathcal O_U\!\left(
  \sum_{i=1}^l \lfloor b_i - a_i(v_j) \rfloor D_i
\right)\!\cdot v_j,
\]
where $D_i := \{z_i = 0\} \subset \Delta^n$ for each $i$.
\end{thm}

Theorem~\ref{x-thm1.1} implies that, for each $i=1,\ldots,l$ and
$b\in(a_i-1,a_i]$, the image
\[
{}^i\!F_b\bigl({}_{\bm a} E|_{D_i}\bigr)
\subset {}_{\bm a} E|_{D_i}
\]
of the natural morphism
${}_{\bm a(i,b)}E|_{D_i} \to {}_{\bm a}E|_{D_i}$
is a subbundle.
Here $\bm a(i,b)$ is defined by replacing the $i$-th component of
$\bm a$ by $b$.
The induced filtrations ${}^i\!F$ $(i=1,\ldots,l)$ on
${}_{\bm a} E|_{D_i}$ are mutually compatible. 
These filtrations are referred 
to as the \emph{parabolic filtrations 
associated with ${}_{\bm a}E$}. 

Let $\mathbf F = ({}^iF \mid i=1,\ldots,l)$ denote the resulting tuple of
filtrations.
Let $\bm v=\{v_1,\ldots,v_r\}$ be a local 
frame of ${}_{\bm a}E$ compatible with
$\mathbf F$ near the origin, and let 
\[ 
a_i(v_k)={}^i\!\deg^{\mathbf F}(v_k)
:=\deg ^{{}^i\!F}(v_k)
\] 
be the
corresponding weights.
We define
\[
v'_k := v_k \cdot \prod_{i=1}^l |z_i|^{a_i(v_k)} .
\]
Let $H(h,\bm v')$ be the Hermitian matrix-valued function whose $(p,q)$-entry
is given by $h(v'_p,v'_q)$.
The following weak norm estimate is a fundamental tool in the study of
acceptable bundles.

\begin{thm}[{Weak norm estimate, cf.~\cite[Theorem~21.3.2]{mochizuki4}}]
\label{x-thm1.4}
There exist positive constants $C$ and $N$ such that, in a neighborhood
of the origin,
\[
C^{-1}
\left(-\sum_{i=1}^l \log |z_i|\right)^{-N} I_r
\le
H(h,\bm v')
\le
C
\left(-\sum_{i=1}^l \log |z_i|\right)^{N} I_r .
\]
Here $I_r$ denotes the identity matrix of size $r$, and 
for Hermitian matrix-valued functions $A$ and $B$,
the notation $A\le B$ means that $B-A$ is positive semidefinite. 
\end{thm}

We can translate various results on acceptable bundles over $\Delta^*$
to the setting of partially punctured polydisks
$(\Delta^*)^l \times \Delta^{n-l}$. 
Below we briefly explain some of these results for the reader's convenience.

\begin{thm}[{Dual 
bundles, see \cite[Theorem 1.12]{fujino-fujisawa-ono}}]\label{x-thm1.5}
Let $(E, h)$ be an acceptable vector bundle on a partially punctured
polydisk $(\Delta^*)^l \times \Delta^{n-l}$.
Then, for any $\bm a \in \mathbb R^l$, we have
\[
\bigl({}_{\bm a} E\bigr)^\vee
=
{}_{-\bm a + \bm 1 - \bm \varepsilon}\bigl(E^\vee\bigr),
\]
where $\bm 1 = (1,\ldots,1) \in \mathbb R^l$ and
$\bm \varepsilon = (\varepsilon,\ldots,\varepsilon) \in \mathbb R^l$ 
with $0<\varepsilon \ll 1$.
\end{thm}

\begin{thm}[{Tensor 
products, see \cite[Theorem 1.14]{fujino-fujisawa-ono}}]\label{x-thm1.6}
Let $(E_1, h_1)$ and $(E_2, h_2)$ be acceptable vector bundles on a
partially punctured polydisk $(\Delta^*)^l \times \Delta^{n-l}$.
Then, for any $\bm b \in \mathbb R^l$, we have
\[
{}_{\bm b}(E_1 \otimes E_2)
=
\sum_{\bm a_1 + \bm a_2 \leq \bm b}
{}_{\bm a_1} E_1 \otimes {}_{\bm a_2} E_2.
\]
\end{thm}

\begin{thm}[{Hom bundles, 
see \cite[Proposition 17.1]{fujino-fujisawa-ono}}]\label{x-thm1.7}
Let $(E_1, h_1)$ and $(E_2, h_2)$ be acceptable vector bundles on a
partially punctured polydisk $(\Delta^*)^l \times \Delta^{n-l}$.
Then, for any $\bm a \in \mathbb R^l$, we have
\[
{}_{\bm a}\!\Hom(E_1, E_2)
=
\bigl\{
f \in \Hom_{\mathcal O_{(\Delta^*)^l \times \Delta^{n-l}}}(E_1, E_2)
\ \bigm|\
f({}_{\bm k} E_1) \subset {}_{\bm a + \bm k} E_2
\text{ for all } \bm k \in \mathbb R^l
\bigr\}.
\]
\end{thm}

As a special case of Theorem~\ref{x-thm1.7}, we obtain the following
statement.

\begin{cor}[{see \cite[Proposition~21.3.3]{mochizuki4}}]\label{x-cor1.8}
Let $(E, h)$ be an acceptable vector bundle on a partially punctured
polydisk $(\Delta^*)^l \times \Delta^{n-l}$.
Then ${}^\diamond\!\End(E)$ is canonically isomorphic to the sheaf of
endomorphisms $f$ of ${}_{\bm a} E$, for any $\bm a \in \mathbb R^l$,
such that $f|_{D_i}$ preserves the filtration ${}^i\!F$
for each $i = 1, \ldots, l$.
\end{cor}

As in \cite{fujino-fujisawa-ono}, we adopt the following convention
throughout this paper.

\begin{say}[Convention]\label{x-say1.9}
Let $\mathcal{F}$ be a sheaf on a topological space $X$.  
Unless explicitly stated otherwise, we write  
$f \in \mathcal{F}$ to indicate that $f$ is a local section  
$f \in \mathcal{F}(U)$ over some open subset $U \subset X$.  

In this paper, we do not distinguish between holomorphic vector bundles  
on a complex manifold $X$ and the corresponding locally free  
$\mathcal{O}_X$-modules. These are treated as equivalent  
unless stated otherwise.
\end{say}

This paper is organized as follows.
In Section~\ref{x-sec2}, we introduce the notion of acceptable bundles on
complex manifolds. 
In Section~\ref{x-sec3}, we 
recall basic notions concerning increasing 
$\mathbb R$-indexed filtrations on vector spaces and vector bundles. 
Section~\ref{x-sec4} is devoted to a brief review of filtered bundles in
the sense of Mochizuki.
In Section~\ref{x-sec5}, we recall basic definitions and properties of
plurisubharmonic functions for the sake of completeness.
Section~\ref{x-sec6} discusses fundamental properties of acceptable
bundles on partially punctured polydisks, and Section~\ref{x-sec7} is
devoted to several preliminary estimates in this setting.
In Section~\ref{x-sec8}, we explain a special case of the
Ohsawa--Takegoshi $L^2$ extension theorem.
Section~\ref{x-sec9} reviews results on acceptable bundles over a
punctured disk, following \cite{fujino-fujisawa-ono}.
In Section~\ref{x-sec10}, we study the behavior of acceptable vector
bundles over a punctured disk under pull-back by cyclic coverings.

Sections~\ref{x-sec11} and~\ref{x-sec12} form the technical core of this
paper.
In these sections, we develop the theory of prolongations of acceptable
line bundles and vector bundles on
$\Delta^*\times \Delta^{n-1}$, which provides the essential ingredients
for the proofs of the main results given in the subsequent sections.
Finally, in Section~\ref{x-sec13}, we prove one of the main results of
this paper, namely Theorem~\ref{x-thm1.1}, and in
Section~\ref{x-sec14} we establish the weak norm estimate stated in
Theorem~\ref{x-thm1.4}.
In the final section, Section~\ref{x-sec15}, 
we establish
Theorems~\ref{x-thm1.5}, \ref{x-thm1.6}, and~\ref{x-thm1.7}, together with
Corollary~\ref{x-cor1.8}, completing the proofs of the remaining results
via a systematic reduction to the curve case.

\begin{ack}\label{x-ack}
The first author was partially 
supported by JSPS KAKENHI Grant Numbers 
JP20H00111, JP21H00974, JP21H04994, JP23K20787. 
The third author was supported by 
JSPS KAKENHI Grant Number JP24KJ1611. 
\end{ack}

\section{Acceptable bundles on a complex manifold}\label{x-sec2}

Although our main interest lies in acceptable bundles on a partially  
punctured polydisk, we begin by recalling the general framework.

\medskip

Let $X$ be a complex manifold with $\dim_{\mathbb{C}} X = n$, and let  
$D = \sum_{i \in I} D_i$ be a simple normal crossing divisor on $X$.

\begin{defn}[Admissible coordinates, {\cite[Definition~4.1]{mochizuki1}}]\label{x-def2.1}
Let $P \in X$, and let $D_{i_j}$ ($j=1,\dots,l$) be the components of $D$ passing through $P$.  
An \emph{admissible coordinate system} around $P$ is a pair $(\mathcal{U},\varphi)$ satisfying:
\begin{itemize}
  \item $\mathcal{U}$ is an open neighborhood of $P$ in $X$;
  \item $\varphi$ is a holomorphic isomorphism  
  \[
     \varphi : \mathcal{U} \xrightarrow{\ \sim\ } \Delta^n:=\{(z_1,\dots,z_n)\mid |z_i|<1\},
  \]
  such that $\varphi(P) = (0,\dots,0)$ and  
  $\varphi(D_{i_j})=\{z_j=0\}$ for each $j=1,\dots,l$.
\end{itemize}
\end{defn}

Let $(E,h)$ be a holomorphic vector bundle on $X\setminus D$ equipped with a smooth Hermitian metric $h$.  
Given a collection of real numbers $\bm{\alpha} = (\alpha_i)_{i\in I}\in \mathbb{R}^I$, we recall the notion of prolongation.

\begin{defn}[Prolongation 
by increasing orders, 
{\cite[Definition~4.2]{mochizuki2}}]\label{x-def2.2}
Let $U \subset X$ be open, and let $s \in \Gamma(U\setminus D, E)$ be a section.  
We say that the increasing order of $s$ is at most $\bm{\alpha}$ if the following holds:
\begin{itemize}
  \item For every $P\in U$, choose an admissible coordinate system $(\mathcal{U},\varphi)$ around $P$.  
  Then for every $\varepsilon>0$ there exists a constant $C>0$ such that on $\mathcal{U}$,
  \[
    |s|_h \le \frac{C}{\prod_{j=1}^l |z_j|^{\alpha_{i_j}+\varepsilon}}.
  \]
\end{itemize}
In this case we write $-\ord(s)\le \bm{\alpha}$.
\end{defn}

For $\bm{\alpha}\in\mathbb{R}^I$, we define an $\mathcal{O}_X$-module ${}_{\bm{\alpha}}E$ by setting
\[
  \Gamma(U, {}_{\bm{\alpha}}E) := 
  \{\, s \in \Gamma(U\setminus (U\cap D), E)\mid -\ord(s)\le \bm{\alpha}\,\}
\] 
for any open subset $U\subset X$. 
The sheaf ${}_{\bm{\alpha}}E$ is called the \emph{prolongation} of $E$ of increasing order $\bm{\alpha}$.

\begin{defn}[Poincar\'e metric]\label{x-def2.3}
On
\[
(\Delta^*)^l \times \Delta^{n-l}
= \{ (z_1,\dots,z_n)\in\mathbb{C}^n \mid |z_i|<1\ \text{for all }i,\ z_j\neq 0\ \text{for }j\le l\},
\]
the \emph{Poincar\'e metric} is defined by
\[
  \omega_P := \sum_{j=1}^l 
  \frac{\sqrt{-1}\,dz_j\wedge d\overline{z}_j}{|z_j|^2(-\log|z_j|^2)^2}
  +\sum_{k=l+1}^n \frac{\sqrt{-1}\,dz_k\wedge d\overline{z}_k}{(1-|z_k|^2)^2}.
\]
Equivalently,
\[
\omega_P = -\sqrt{-1}\,\partial\overline{\partial}
 \log\!\left(
    \prod_{j=1}^l (-\log|z_j|^2)\,
    \prod_{k=l+1}^n (1-|z_k|^2)
 \right).
\]
\end{defn}

Let $P\in X$, and choose an admissible coordinate system $(\mathcal{U},\varphi)$ around $P$.  
Via the isomorphism
\[
\varphi \colon \mathcal{U}\setminus D \xrightarrow{\sim} (\Delta^*)^l \times \Delta^{n-l},
\]
we pull back the Poincar\'e metric to obtain a Hermitian metric $g_{\mathbf{P}}$ on $\mathcal{U}\setminus D$.

Given the Hermitian metric $h$ on $E$ and the Poincaré metric $g_{\mathbf{P}}$ on $T_{\mathcal{U}\setminus D}$,  
we equip $\Hom(E, E)\otimes \Omega^{p,q}$ with the induced Hermitian metric 
$(\cdot,\cdot)_{h,g_{\mathbf{P}}}$ on $\mathcal{U}\setminus D$.

\begin{defn}[Acceptable bundles, {\cite[Definition~4.3]{mochizuki1}}]\label{x-def2.4}
Let $(E,h)$ be a holomorphic vector bundle on $X \setminus D$ equipped with a 
smooth Hermitian metric $h$.  
Let $D_h = D'_h + \bar\partial$ denote its \emph{Chern connection}. 
The curvature form of $(E,h)$ is defined by
\[
  \sqrt{-1}\,\Theta_h(E) := \sqrt{-1}\,D_h^2,
\]
which is a smooth $\Hom(E,E)$-valued $(1,1)$-form on $X \setminus D$.

We say that $(E,h)$ is \emph{acceptable at $P$} if, for an admissible coordinate system 
$(\mathcal{U},\varphi)$ around $P$, the norm of the curvature 
$\sqrt{-1}\,\Theta_h(E)$ with respect to $(\cdot,\cdot)_{h,g_{\mathbf{P}}}$ is bounded on 
$\mathcal{U} \setminus D$.

If $(E,h)$ is acceptable at every point of $X$, then we simply call it \emph{acceptable}.
\end{defn}

\section{On Filtrations}\label{x-sec3}

In this section, we recall basic notions concerning increasing
$\mathbb R$-indexed filtrations on vector spaces and vector bundles,
following \cite{mochizuki1} and \cite{mochizuki3}.

\begin{defn}[{cf.~\cite[Definition 4.1]{mochizuki3}}]\label{x-def3.1}
Let $V$ be a finite-dimensional vector space.
An ({\em{increasing}}) {\em{filtration}} $F$ of $V$ indexed by $\mathbb R$
is a family of subspaces
\[
\{F_\eta \mid \eta\in\mathbb R\}
\]
satisfying the following conditions:
\begin{itemize}
\item $F_\eta \subset F_{\eta'}$ for $\eta \le \eta'$;
\item $F_\eta = V$ for any sufficiently large $\eta$.
\end{itemize}

When considering a tuple of filtrations, we write
\[
\mathbf F = ({}^i\!F \mid i\in I).
\]
For each $i\in I$, we denote by ${}^i\!\Gr^{\mathbf F}$ the graded space
$\Gr^{{}^i\!F}$.

For a nonzero vector $v\in V$, we define
\[
\deg^F(v) := \min\{\eta\in\mathbb R \mid v\in F_\eta\}.
\]

A basis $\bm v=\{v_1,\ldots,v_r\}$ of $V$ is said to be
\emph{compatible with the filtration $F$} if there exists a decomposition
\[
\bm v = \bigsqcup_{\lambda\in\mathbb R} \bm v_\lambda
\]
such that, for each $\lambda\in\mathbb R$,
the subset $\bm v_\lambda$ consists of vectors contained in $F_\lambda$
and induces a basis of the graded piece $\Gr^F_\lambda V$.
\end{defn}

\begin{defn}[{cf.~\cite[Definition 4.2]{mochizuki3}}]\label{x-def3.2}
Let
\[
\mathbf F = ({}^i\!F \mid i\in I)
\]
be a tuple of $\mathbb R$-indexed filtrations of $V$.
The tuple $\mathbf F$ is said to be \emph{compatible} if there exists
a direct sum decomposition
\[
V = \bigoplus_{\bm\eta\in\mathbb R^I} U_{\bm\eta}
\]
such that
\begin{equation}\label{x-eq3.1}
{}^I\!F_{\bm\rho}
:= \bigcap_{i\in I} {}^i\!F_{\rho_i}
= \bigoplus_{\bm\eta\le \bm\rho} U_{\bm\eta}
\end{equation}
for all $\bm\rho=(\rho_i)_{i\in I}\in\mathbb R^I$.
Here, $\bm\eta\le\bm\rho$ means $\eta_i\le\rho_i$ for all $i\in I$.

Any decomposition satisfying \eqref{x-eq3.1} is called a
\emph{splitting} of the compatible tuple $\mathbf F$.
\end{defn}

From now on, we consider filtrations on vector bundles.

\begin{defn}[{cf.~\cite[Definition 4.8]{mochizuki3}}]\label{x-def3.3}
Let $X$ be a complex manifold and let $V$ be a vector bundle on $X$.
A filtration $F$ of $V$ indexed by $\mathbb R$
is a family of subbundles  
\[
\{F_\eta \subset V \mid \eta\in\mathbb R\}
\]
such that $F_\eta\subset F_{\eta'}$ for $\eta\le\eta'$ and
$F_\eta=V$ for $\eta\gg 0$.

Let
\[
\mathbf F = ({}^i\!F \mid i\in I)
\]
be a tuple of filtrations of $V$.
For a point $P\in X$, the induced tuple of filtrations on the fiber
$V|_P$ is denoted by $\mathbf F|_P$.
\end{defn}

To treat parabolic filtrations, we introduce the following notions.

\begin{defn}[{cf.~\cite[Definition 3.12]{mochizuki2}}]\label{x-def3.4}
Let $X$ be a complex manifold and let $V$ be a vector bundle on $X$.
Let
\[
Y=\sum_{i\in I} Y_i
\]
be a simple normal crossing divisor on $X$.
For each $i\in I$, let ${}^i\!F$ be a filtration of $V|_{Y_i}$
in the sense of Definition~\ref{x-def3.3}.
The tuple of filtrations
\[
\mathbf F=\bigl({}^i\!F \mid i\in I\bigr)
\]
is said to be \emph{compatible} if, for any subset $J\subset I$,
there exists, locally on
\[
Y_J:=\bigcap_{j\in J} Y_j,
\]
a direct sum decomposition
\[
V|_{Y_J}=\bigoplus_{\bm\eta\in\mathbb R^J} U_{\bm\eta}
\]
such that
\[
{}^J\!F_{\bm\rho}
:=\bigcap_{j\in J} {}^j\!F_{\rho_j}\big|_{Y_J}
=
\bigoplus_{\bm\eta\le \bm\rho} U_{\bm\eta}
\]
holds for all $\bm\rho\in\mathbb R^J$.
\end{defn}

\begin{defn}[{cf.~\cite[Definition 2.16]{mochizuki1}}]\label{x-def3.5}
Let $X$ be a complex manifold, $V$ a vector bundle on $X$,
and $Y\subset X$ a complex submanifold.
Let $F$ be a filtration of $V|_Y$ in the sense of
Definition~\ref{x-def3.3}.

A smooth section $f$ of $V$ is said to be
\emph{compatible with the filtration $F$} if the value
\[
\deg^{F|_P}(f(P))
\]
is independent of the point $P\in Y$.
In this case, we define
\[
\deg^F(f) := \deg^{F|_P}(f(P))
\]
for any $P\in Y$.
\end{defn}

\begin{defn}[{cf.~\cite[Definition 2.17]{mochizuki1}}]\label{x-def3.6}
Let $\bm v=\{v_1,\ldots,v_r\}$ be a smooth frame of $V$.
It is said to be \emph{compatible with the filtration $F$ of $V|_Y$} if
the following conditions are satisfied:
\begin{itemize}
\item[(1)] Each $v_i$ is compatible with $F$ in the sense of
Definition~\ref{x-def3.5};
\item[(2)] For any point $P\in Y$, the frame $\bm v|_P$
is compatible with the filtration $F|_P$
in the sense of Definition~\ref{x-def3.1}. 
\end{itemize} 

Let
\[
Y=\sum_{i\in I} Y_i
\]
be a simple normal crossing divisor on $X$.
For each $i\in I$, let ${}^i\!F$ be a filtration of $V|_{Y_i}$.
The smooth frame $\bm v=\{v_1,\ldots,v_r\}$ of $V$ is said to be
\emph{compatible with the tuple of filtrations}
\[
\mathbf F=\bigl({}^i\!F \mid i\in I\bigr)
\]
if $\bm v$ is compatible with ${}^i\!F$ for every $i\in I$. 
\end{defn}

\section{Filtered bundles}\label{x-sec4}

In this section, we briefly review the notion of filtered bundles in the
sense of Mochizuki, following 
\cite[2.3 Filtered bundles]{mochizuki5}. 

\begin{defn}[{Filtered bundles in the local case, cf.~\cite[2.3.1]{mochizuki5}}]\label{x-def4.1}
Let $U$ be an open neighborhood of $(0,\ldots,0)$ in $\mathbb C^n$.
We set
\[
D_{U,i} := U \cap \{ z_i = 0 \}, \qquad
D_U := \bigcup_{i=1}^l D_{U,i},
\]
where $1 \le l \le n$.
Let $\mathcal V$ be a locally free $\mathcal O_U(\ast D_U)$-module.

A \emph{filtered bundle} $\mathcal P_\ast \mathcal V$ of $\mathcal V$
is a family of locally free $\mathcal O_U$-submodules
$\mathcal P_{\bm a} \mathcal V$ indexed by $\bm a \in \mathbb R^l$
satisfying the following conditions:
\begin{itemize}
\item[(1)]
If $\bm a \le \bm b$ (i.e., $a_i \le b_i$ for all $i=1,\ldots,l$), then
\[
\mathcal P_{\bm a} \mathcal V \subset \mathcal P_{\bm b} \mathcal V .
\]

\item[(2)]
There exists a frame
$\bm v = \{v_1,\ldots,v_r\}$ of $\mathcal V$
and vectors $\bm a(v_j) \in \mathbb R^l$ ($j=1,\ldots,r$) such that,
for any $\bm b \in \mathbb R^l$, we have
\begin{equation}\label{x-eq4.1}
\mathcal P_{\bm b} \mathcal V
=
\bigoplus_{j=1}^r
\mathcal O_U \!\left(
\sum_{i=1}^l
\lfloor b_i - a_i(v_j) \rfloor \, D_{U,i}
\right)
\cdot v_j .
\end{equation}
\end{itemize}
\end{defn}

Let $X$ be a complex manifold with a simple normal crossing divisor $D$.
Let
\[
D = \bigcup_{i \in \Lambda} D_i
\]
be the irreducible decomposition of $D$.
For any point $P \in D$, a holomorphic coordinate neighborhood
$(X_P, z_1, \ldots, z_n)$ around $P$ is called \emph{admissible} 
(see Definition \ref{x-def2.1}) if
\[
D_P := D \cap X_P = \bigcup_{i=1}^{l(P)} \{ z_i = 0 \}.
\]

For such an admissible coordinate neighborhood, there exists a uniquely
determined map
\[
\rho_P \colon \{1,\ldots,l(P)\} \longrightarrow \Lambda
\]
such that
\[
D_{\rho_P(i)} \cap X_P = \{ z_i = 0 \}.
\]
We define a map
\[
\kappa_P \colon \mathbb R^{\Lambda} \longrightarrow \mathbb R^{l(P)}
\]
by
\[
\kappa_P(\bm a) :=
\bigl( a_{\rho_P(1)}, \ldots, a_{\rho_P(l(P))} \bigr).
\]

\begin{defn}[{Filtered bundles, cf.~\cite[2.3.3]{mochizuki5}}]
\label{x-def4.2}
Let $\mathcal V$ be a locally free $\mathcal O_X(\ast D)$-module.
A \emph{filtered bundle}
\[
\mathcal P_\ast \mathcal V
=
\left(
\mathcal P_{\bm a} \mathcal V \mid \bm a \in \mathbb R^{\Lambda}
\right)
\]
of $\mathcal V$ is a family of locally free $\mathcal O_X$-submodules
$\mathcal P_{\bm a} \mathcal V \subset \mathcal V$
satisfying the following conditions:
\begin{itemize}
\item[(1)]
For any $P \in D$, take an admissible coordinate neighborhood
$(X_P, z_1, \ldots, z_n)$ around $P$.
Then, for any $\bm a \in \mathbb R^{\Lambda}$,
the restriction $\mathcal P_{\bm a} \mathcal V|_{X_P}$
is determined only by $\kappa_P(\bm a)$.
We denote it by
\[
\mathcal P^{(P)}_{\kappa_P(\bm a)}
\bigl( \mathcal V|_{X_P} \bigr).
\]

\item[(2)]
The family
\[
\left(
\mathcal P^{(P)}_{\bm b}
\bigl( \mathcal V|_{X_P} \bigr)
\ \middle|\ 
\bm b \in \mathbb R^{l(P)}
\right)
\]
is a filtered bundle over $\mathcal V|_{X_P}$ in the sense of
Definition~\ref{x-def4.1}.
\end{itemize}
\end{defn}

For any subset $I \subset \Lambda$, let $\bm\delta_I \in \mathbb R^\Lambda$
be the element whose $j$-th component is $1$ for $j \in I$ and $0$ for
$j \in \Lambda \setminus I$.
We set
\[
D_I := \bigcap_{i \in I} D_i,
\qquad
\partial D_I :=
D_I \cap \left( \bigcup_{j \in \Lambda \setminus I} D_j \right).
\]

Let $\mathcal P_\ast \mathcal V$ be a filtered bundle on $(X,D)$.
Fix $i \in \Lambda$ and $\bm a \in \mathbb R^\Lambda$.
For any $b$ satisfying $a_i-1 \le b \le a_i$, we set
\[
\bm a(b,i) := \bm a + (b-a_i)\bm\delta_i .
\]
We define
\[
{}^i\!F_b
\bigl( \mathcal P_{\bm a} \mathcal V|_{D_i} \bigr)
:=
\mathcal P_{\bm a(b,i)} \mathcal V
\big/
\mathcal P_{\bm a(a_i-1,i)} \mathcal V .
\]
It is naturally a locally free $\mathcal O_{D_i}$-module and can be
regarded as a subbundle of
$\mathcal P_{\bm a} \mathcal V|_{D_i}$.
In this way, we obtain a filtration ${}^i\!F$ of
$\mathcal P_{\bm a} \mathcal V|_{D_i}$ indexed by the interval
$(a_i-1,a_i]$.
If there is no risk of confusion, we simply write $F$.

For $I \subset \Lambda$ and $i \in I$, the filtrations ${}^i\!F$ induce a
filtration on $\mathcal P_{\bm a} \mathcal V|_{D_I}$.
Let $\bm a_I \in \mathbb R^I$ be the image of $\bm a$ under the natural
projection $\mathbb R^\Lambda \to \mathbb R^I$, and set
\[
(\bm a_I - \bm\delta_I, \bm a_I]
:= \prod_{i \in I} (a_i-1,a_i].
\]
For any $\bm b \in (\bm a_I - \bm\delta_I, \bm a_I]$, we define
\[
{}^I\!F_{\bm b}
\bigl( \mathcal P_{\bm a} \mathcal V|_{D_I} \bigr)
:=
\bigcap_{i \in I}
{}^i\!F_{b_i}
\bigl( \mathcal P_{\bm a} \mathcal V|_{D_I} \bigr).
\]

By Definition~\ref{x-def4.1}, the following compatibility holds.
\begin{itemize}
\item
Let $P$ be a point of $D_I$.
There exists an open neighborhood $X_P$ of $P$ in $X$ and a
(non-canonical) decomposition
\[
\mathcal P_{\bm a} \mathcal V|_{D_I \cap X_P}
=
\bigoplus_{\bm b \in (\bm a_I-\bm\delta_I,\bm a_I]}
\mathcal G_{P,\bm b}
\]
such that, for any
$\bm c \in (\bm a_I-\bm\delta_I,\bm a_I]$, we have
\begin{equation}\label{x-eq4.2}
{}^I\!F_{\bm c}
\bigl( \mathcal P_{\bm a} \mathcal V|_{D_I \cap X_P} \bigr)
=
\bigoplus_{\bm b \le \bm c}
\mathcal G_{P,\bm b}.
\end{equation}

Indeed, there exists a frame
$\bm v=\{v_1,\ldots,v_r\}$ of $\mathcal P_{\bm a} \mathcal V$ around $P$
with tuples $\bm a(v_j) \in \mathbb R^{l(P)}$ of real numbers satisfying
\eqref{x-eq4.1}, where $\bm b$ is replaced by $\bm a$.
There exists a bijection
\[
\kappa \colon I \simeq \{1,\ldots,l(P)\}
\]
determined by
$D_i \cap X_P = \{ z_{\kappa(i)} = 0 \}$,
by which we identify $I$ with $\{1,\ldots,l(P)\}$.
Let $\mathcal G_{P,\bm b}$ be the subbundle of
$\mathcal P_{\bm a} \mathcal V|_{D_I \cap X_P}$ generated by
$v_j|_{D_I \cap X_P}$ with $\bm a(v_j)=\bm b$.
Then \eqref{x-eq4.2} follows.
\end{itemize}

For any $\bm c \in (\bm a_I-\bm\delta_I,\bm a_I]$, we define a locally free
$\mathcal O_{D_I}$-module
\[
{}^I\!\Gr^F_{\bm c}
\bigl( \mathcal P_{\bm a} \mathcal V \bigr)
:=
\frac{
{}^I\!F_{\bm c}
\bigl( \mathcal P_{\bm a} \mathcal V|_{D_I} \bigr)
}{
\sum_{\bm b \lneq \bm c}
{}^I\!F_{\bm b}
\bigl( \mathcal P_{\bm a} \mathcal V|_{D_I} \bigr)
},
\]
where $\bm b=(b_i)\lneq \bm c=(c_i)$ means that $b_i\le c_i$ for all $i$
and $\bm b\neq\bm c$.

\section{Plurisubharmonic functions}\label{x-sec5}

For the sake of completeness, we recall the definition of plurisubharmonic functions, which play an important role throughout this paper. 

\begin{defn}[Plurisubharmonic functions]\label{x-def5.1}
Let $\Omega$ be an open subset of $\mathbb C^n$. 
A function $u\colon \Omega \to [-\infty, +\infty)$ is said to be {\em plurisubharmonic} ({\em psh}, for short) if 
\begin{itemize}
\item[(1)] $u$ is upper semicontinuous, and 
\item[(2)] for every complex line $L \subset \mathbb C^n$, the restriction $u|_{\Omega \cap L}$ is subharmonic on $\Omega \cap L$; that is, for all $a \in \Omega$ and $\xi \in \mathbb C^n$ with $|\xi| < d(a, \Omega^c)$, 
\[
u(a) \leq \frac{1}{2\pi} \int_0^{2\pi} u(a+\xi e^{\sqrt{-1}\theta})\, d\theta.
\]
\end{itemize}
\end{defn}

For the basic properties of plurisubharmonic (psh, for short) functions, see, for 
example, \cite[1.B.~Plurisubharmonic Functions]{demailly1} and 
\cite[3.3~Plurisubharmonic Functions]{noguchi-ochiai}. 
In this paper, the notion of the Lelong number plays a crucial role, so we recall it here for the reader's convenience. 
For further details, see, for example, \cite[2.B.~Lelong Numbers]{demailly1}. 

\begin{defn}[Lelong numbers]\label{x-def5.2} 
Let $u$ be a plurisubharmonic (psh) function on an open subset $\Omega \subset \mathbb C^n$. 
Then $\sqrt{-1}\,\partial \overline\partial u$ defines a closed positive $(1,1)$-current on $\Omega$, hence determines a positive Radon measure. 
The {\em Lelong number} $\nu(u, x)$ of $u$ at a point $x \in \Omega$ is defined by
\[
\nu(u, x):= \liminf_{z\to x} \frac{u(z)}{\log |z-x|} \in \mathbb R_{\ge 0}.
\]
It is well known that
\begin{equation}\label{x-eq5.1}
\nu(u, x) = \lim_{r\to +0} \frac{1}{r^{2(n-1)}} 
\int_{B(x, r)} \frac{\sqrt{-1}}{\pi}\, \partial \overline\partial u 
\wedge \left( \frac{\sqrt{-1}}{2\pi} \sum_{i=1}^n dz_i \wedge d\overline z_i \right)^{n-1}
\end{equation} 
holds, 
where $B(x, r) = \{ z \in \mathbb C^n \mid |z-x| < r \}$.  

By Siu's theorem (see, for example, \cite[(13.3) Corollary]{demailly1}), 
for every $c > 0$, the {\em upper level set of the Lelong number}
\[
E_c(u) := \{ z \in \Omega \mid \nu(u, z) \ge c \}
\]
is a closed analytic subset of $\Omega$. 
\end{defn}

In Definition~\ref{x-def5.2}, the right-hand side of the equality~\eqref{x-eq5.1} 
is the original definition of the Lelong number (see, for example, 
\cite[Chapter~III, (5.7)]{demailly}).  
Although the equality in~\eqref{x-eq5.1} is not at all 
obvious, it is a well-known fact 
(see, for example, \cite[Chapter~III, (6.9), Example]{demailly}).  
A relatively accessible proof of Siu's theorem appearing in 
Definition~\ref{x-def5.2} can be found in 
\cite[13.A.~Approximation of Plurisubharmonic Functions 
via Bergman kernels]{demailly1}.  
It is a particularly elegant application of the Ohsawa--Takegoshi 
$L^2$-extension theorem.

\begin{lem}\label{x-lem5.3}
Let $u$ be a plurisubharmonic function on an open subset 
$\Omega \subset \mathbb C^n$, and let $x \in \Omega$ be such that $\overline B(x, R_0)=\{z\in \mathbb C^n\mid 
|z-x|\leq R_0\} \subset \Omega$. 
Then the function
\[
\log r \longmapsto \sup_{|z-x|=r} u(z)
\]
is convex and nondecreasing for $0 \le r \le R_0$. 
\end{lem}

\begin{proof}[Proof of Lemma \ref{x-lem5.3}]
Since 
\[
\sup_{|z-x|=r} u(z) = \max_{z\in \overline B(x, r)} u(z),
\]
it is clear that $\sup_{|z-x|=r} u(z)$ is a nondecreasing function of $r$ 
(see, for example, \cite[(3.3.2) Theorem and (3.3.27) Remark]{noguchi-ochiai}). 

For each $\zeta \in \mathbb C^n$ with $|\zeta| = 1$, the function 
\[
\mathbb C \ni w \longmapsto u(x + \zeta w)
\]
is subharmonic by definition. 
Hence 
\[
\mathbb C \ni w \longmapsto u(x + \zeta e^w)
\]
is also subharmonic (see, for example, \cite[(1.8) Proposition]{demailly1} or 
\cite[(3.3.19) Theorem and (3.3.38) Remark]{noguchi-ochiai}). 
It is easy to check that 
\[
\mathbb C \ni w \longmapsto \sup_{|\zeta|=1} u(x + \zeta e^w)
\]
is upper semicontinuous and locally bounded from above. 
Therefore, by \cite[(3.3.3) Lemma (ii)]{noguchi-ochiai} or 
\cite[Chapter~I, (5.7) Theorem]{demailly}, it is also subharmonic. 
Since this function depends only on $\re w$, it follows that 
$\sup_{|\zeta|=1} u(x + \zeta e^w)$ is convex as a function of $\re w$. 
This proves the lemma. 
\end{proof}

The following lemma is an elementary 
property of convex functions and is included here for completeness.

\begin{lem}\label{x-lem5.4}
Let $f\colon (-\infty, b]\to \mathbb R$ be a convex function such that 
\[
\lim_{x\to -\infty}\frac{f(x)}{x}=\nu \in \mathbb R. 
\] 
Then 
\[
f(x) \le \nu (x-b) + f(b)
\]
for all $x \in (-\infty, b]$. 
\end{lem}

\begin{proof}[Proof of Lemma \ref{x-lem5.4}]
Since $f$ is convex, we have 
\[
f(\lambda x_1+(1-\lambda)x_2)\le \lambda f(x_1)+(1-\lambda) f(x_2)
\]
for all $x_1, x_2 \in (-\infty, b]$ and $\lambda \in [0,1]$.  
This implies that the difference quotient
\[
\frac{f(b)-f(x)}{b-x}
\]
is nondecreasing in $x$ on $(-\infty, b)$.  
In particular, for any $x_1, x_2 \in (-\infty, b)$ with $x_1 \le x_2$, we have
\[
\frac{f(b)-f(x_1)}{b-x_1} \le \frac{f(b)-f(x_2)}{b-x_2}.
\]

By the assumption 
\[
\lim_{x\to -\infty}\frac{f(x)}{x} = \nu,
\]
we obtain
\[
\lim_{x\to -\infty}\frac{f(b)-f(x)}{b-x} 
= \lim_{x\to -\infty}\frac{f(x)}{x} = \nu.
\]
Hence, for every $x \in (-\infty, b)$,
\[
\frac{f(b)-f(x)}{b-x} \ge \nu.
\]
Multiplying both sides by $b-x > 0$, we get
\[
f(x) \le \nu(x-b) + f(b),
\]
which proves the lemma. 
\end{proof}

By Lemma~\ref{x-lem5.3} and Lemma~\ref{x-lem5.4}, we 
obtain the following corollary.

\begin{cor}\label{x-cor5.5}
Let $u$ be a plurisubharmonic function on an open subset 
$\Omega \subset \mathbb C^n$, and let $x \in \Omega$ satisfy $\overline B(x, R_0) \subset \Omega$. 
Then 
\[
u(z) \le \nu(u, x)\, \log \frac{|z-x|}{R_0} + \max_{w\in \overline B(x, R_0)} u(w)
\]
for all $z \in \overline B(x, R_0)$. 
\end{cor}

\begin{proof}[Proof of Corollary \ref{x-cor5.5}]
Note that 
\[
\nu(u, x) = \lim_{r \to +0} \frac{\sup_{|z-x|=r} u(z)}{\log r}.
\]
By Lemma \ref{x-lem5.3}, the function $\sup_{|z-x|=r} u(z)$ is convex and 
nondecreasing in $\log r$. 
The desired inequality follows immediately from this fact by 
Lemma \ref{x-lem5.4}. 
\end{proof}

We will need the following easy lemma in later sections. 

\begin{lem}\label{x-lem5.6} 
On the polydisk $\Delta^n$, consider 
\[
\chi(0, N) = -N\left( \log(-\log |z_1|^2) + \sum_{k=2}^n \log(1 - |z_k|^2) \right),
\]
where $N > 0$. 
Then both $\chi(0, N)$ and $\log |z_1|^2$ are plurisubharmonic functions. 

For any $Q \in \Delta^{n-1}$, let $Q' = (0, Q)$. 
Then 
\[
\nu(\chi(0, N), Q') = 0 \quad \text{and} \quad  
\nu(\log |z_1|^2, Q') = 2.
\]
\end{lem}

\begin{proof}[Proof of Lemma \ref{x-lem5.6}]
Note that $\chi(0, N)$ is smooth on $\Delta^* \times \Delta^{n-1}$. 
A direct computation shows that
\[
\sqrt{-1}\,\partial \overline\partial \chi(0, N) = N \omega_P.
\]
Thus $\chi(0, N)$ is plurisubharmonic on $\Delta^* \times \Delta^{n-1}$. 
We define $\chi(0, N) \equiv -\infty$ on $\{0\} \times \Delta^{n-1}$; 
then $\chi(0, N)$ extends to a plurisubharmonic function on $\Delta^n$. 
For details, see \cite[(3.3.41) Theorem]{noguchi-ochiai} and 
\cite[Chapter~I, (5.24) Theorem]{demailly}. 

It is easy to see that
\[
0 \le \nu(\chi(0, N), Q') \le 
\liminf_{z_1 \to 0} 
\frac{-N \log(-\log |z_1|^2)}{\log |z_1|} = 0,
\]
hence $\nu(\chi(0, N), Q') = 0$. 

Since $z_1$ is holomorphic on $\Delta^n$, $\log |z_1|^2 = 2\log |z_1|$ is plurisubharmonic. 
It is well known that 
\[
\nu(\log |z_1|^2, Q') = 2\, \operatorname{ord}_{Q'}(z_1) = 2
\]
(see, for example, \cite[(2.8) Theorem (b)]{demailly1}). 
This completes the proof. 
\end{proof}

To make use of Siu's theorem in Definition~\ref{x-def5.2}, 
we prepare the following elementary lemma. 

\begin{lem}\label{x-lem5.7}
Let $V$ be a connected complex manifold and let $f$ be a real-valued function on $V$. 
Assume that for every $a, b \in \mathbb R$, the sets
\[
V_{\ge a} := \{ x \in V \mid f(x) \ge a \}, 
\qquad
V_{\le b} := \{ x \in V \mid f(x) \le b \}
\]
are closed analytic subsets of $V$. 
Then $f$ is constant on $V$. 
\end{lem}

\begin{proof}[Proof of Lemma \ref{x-lem5.7}]
Take $c \in f(V) \subset \mathbb R$. 
For every $\varepsilon > 0$, set
\[
V_c^\varepsilon := V_{\ge c - \varepsilon} \cap V_{\le c + \varepsilon}.
\]
Then 
\[
V = V_{\ge c + \varepsilon} \cup V_c^\varepsilon \cup V_{\le c - \varepsilon}.
\]
Since $c \in f(V)$, both $V_{\ge c + \varepsilon}$ 
and $V_{\le c - \varepsilon}$ are closed analytic subsets 
of $V$ with $V_{\ge c + \varepsilon}\subsetneq V$ 
and $V_{\le c - \varepsilon}\subsetneq V$, hence $V_c^\varepsilon = V$. 
Thus 
\[
V = \bigcap_{\varepsilon > 0} V_c^\varepsilon = \{ x \in V \mid f(x) = c \},
\]
which means that $f$ is constant on $V$. 
\end{proof}

\section{Basic properties of acceptable bundles on 
$(\Delta^*)^l \times \Delta^{n-l}$}\label{x-sec6}

In this section, we discuss basic properties of acceptable bundles on a
polydisk punctured in the first $l$ coordinates. 
A detailed description of acceptable bundles on a partially punctured polydisk is indispensable for the study of acceptable bundles on complex manifolds (see Definition \ref{x-def2.4}). 
We employ the following definition of acceptable vector bundles on a
partially punctured polydisk throughout the present paper.

\begin{defn}[Acceptable bundles on a partially punctured polydisk]
\label{x-def6.1}
Let $E$ be a holomorphic vector bundle on
$(\Delta^*)^l \times \Delta^{n-l}$, equipped with a smooth Hermitian
metric $h$.
We say that $(E,h)$ is \emph{acceptable} if its curvature
$\sqrt{-1}\Theta_h(E)$, viewed as a smooth
$\Hom(E,E)$-valued $(1,1)$-form on
$(\Delta^*)^l \times \Delta^{n-l}$, is bounded with respect to the
Hermitian metric $(\cdot,\cdot)_{h,\omega_P}$, 
which is the natural Hermitian metric on
$\Hom(E,E)\otimes \Omega^{1,1}$ induced by the metric $h$ on $E$
and the Poincar\'e metric $\omega_P$.
In other words, there exists a constant $C>0$ such that
\[
|\sqrt{-1}\Theta_h(E)|_{h,\omega_P} \leq C
\qquad
\text{on } (\Delta^*)^l \times \Delta^{n-l}.
\]
\end{defn}

The following lemma is immediate.

\begin{lem}\label{x-lem6.2}
Let $(E,h)$ be an acceptable vector bundle on
$(\Delta^*)^l \times \Delta^{n-l}$.
Then the dual bundle $(E^\vee,h^\vee)$ and the determinant line bundle
$(\det E, \det h)$ are also acceptable.

Let $(E_1,h_1)$ and $(E_2,h_2)$ be acceptable vector bundles on
$(\Delta^*)^l \times \Delta^{\,n-l}$.
Then the tensor product $(E_1 \otimes E_2, h_1 \otimes h_2)$ and the Hom bundle
$(\Hom(E_1,E_2), h_1^\vee \otimes h_2)$ are acceptable.
\end{lem}

\begin{proof}[Proof of Lemma \ref{x-lem6.2}]
The same argument as in \cite[Lemma~2.2]{fujino-fujisawa-ono} applies
verbatim in our setting.
\end{proof}

We now recall various notions of positivity for vector bundles on a complex
manifold.
For details, see for example \cite[Chapter~10]{demailly1} and 
\cite[Chapter~VII, \S6]{demailly}.

\begin{defn}\label{x-def6.3}
Let $E$ be a holomorphic vector bundle on a complex manifold $X$, equipped with
a smooth Hermitian metric $h$.
Let $D_h$ be the Chern connection of $(E,h)$, and denote the curvature by
$\Theta_h(E)=D^2_h$.

Fix $x\in X$, and choose a frame $e_1,\dots,e_r$ of $E$ at $x$ with dual frame
$e^1,\dots,e^r$.
Let $(z_1,\dots,z_n)$ be local holomorphic coordinates centered at $x$.
Then we may write
\[
\sqrt{-1}\Theta_h(E)
 = \sum_{1 \le j,k \le n}\sum_{1 \le \alpha,\beta \le r}
   R^\beta_{j\overline{k}\alpha}\,
   dz_j \wedge d\overline{z}_k \otimes e^\alpha \otimes e_\beta.
\]
We put
\[
R_{j\overline{k}\alpha\overline{\beta}}
 := h_{\gamma\overline{\beta}}\, R^\gamma_{j\overline{k}\alpha},
 \qquad h_{\gamma\overline{\beta}} := h(e_\gamma,e_\beta).
\]

We say that $(E,h)$ is \emph{Nakano positive} (resp.\ \emph{Nakano semipositive})
at $x$ if
\[
\sum_{j,k,\alpha,\beta}
  R_{j\overline{k}\alpha\overline{\beta}}\, 
  u^{j\alpha}\,\overline{u^{k\beta}}
>0 \quad (\text{resp.\ $\ge 0$})
\]
for any nonzero vector
\[
u = \sum_{j,\alpha} u^{j\alpha}\,
     \left(\frac{\partial}{\partial z_j}\right)\otimes e_\alpha
   \in (T^{1,0}_X \otimes E)_x.
\]

We say that $(E,h)$ is \emph{Griffiths positive} (resp.\ \emph{Griffiths
semipositive}) at $x$ if
\[
\sum_{j,k,\alpha,\beta}
  R_{j\overline{k}\alpha\overline{\beta}}\,
  \xi^j\zeta^\alpha\,\overline{\xi^k}\overline{\zeta^\beta}
>0 \quad (\text{resp.\ $\ge 0$})
\]
for all nonzero
\[
\xi=\sum_j \xi^j \left(\frac{\partial}{\partial z_j}\right)
   \in T^{1,0}_{X,x},
\qquad
\zeta=\sum_\alpha \zeta^\alpha e_\alpha \in E_x.
\]

If $(E, h)$ is Nakano positive (resp.\ Nakano semipositive,
Griffiths positive, or Griffiths semipositive) at every point $x \in X$,
then we simply say that $(E, h)$ is Nakano positive (resp.\ Nakano semipositive,
Griffiths positive, or Griffiths semipositive). 

The notions of Nakano (semi)negativity and Griffiths (semi)negativity are
defined similarly by reversing the inequalities.
\end{defn}

\begin{rem}\label{x-rem6.4}
By definition, Nakano (semi)positivity (resp.\ Nakano (semi)negativity)
implies Griffiths (semi)positivity (resp.\ Griffiths (semi)negativity).
The converse holds when $\dim X=1$ or $\rank E=1$.
\end{rem}

Lemma \ref{x-lem6.5} is a key estimate.

\begin{lem}\label{x-lem6.5}
Let $(E,h)$ be a holomorphic vector bundle on
$X^*=(\Delta^*)^l\times \Delta^{n-l}$.
If
\[
|\sqrt{-1}\Theta_h(E)|_{h,\omega_P} \le C,
\]
that is, if $(E,h)$ is acceptable on $(\Delta^*)^l\times\Delta^{n-l}$,
then
\[
- C\,\omega_P \otimes \Id_E
   \ \le_{\Nak}\ 
\sqrt{-1}\Theta_h(E)
   \ \le_{\Nak}\ 
C\,\omega_P \otimes \Id_E
\]
on $(\Delta^*)^l\times\Delta^{n-l}$.
Here $A \le_{\Nak} B$ means that $B-A$ defines a Nakano semipositive Hermitian 
form on $T^{1,0}_{X^*}\otimes E$ with respect to $h$.
\end{lem}

\begin{proof}[Proof of Lemma \ref{x-lem6.5}]
Fix $x \in (\Delta^*)^l\times\Delta^{n-l}$.
Choose local coordinates $(w_1,\dots,w_n)$ centered at $x$ such that
\[
\omega_P = \sqrt{-1}\sum_{i=1}^n dw_i \wedge d\overline{w}_i.
\]
Choose a holomorphic frame $e_1,\dots,e_r$ of $E$ which is orthonormal at $x$.
Then
\[
\sqrt{-1}\Theta_h(E)
 = \sum_{j,k,\alpha,\beta}
   R^\beta_{j\overline{k}\alpha}\,
   dw_j \wedge d\overline{w}_k \otimes e^\alpha \otimes e_\beta,
\]
and thus $R_{j\overline{k}\alpha\overline{\beta}}(x)
 = R^\beta_{j\overline{k}\alpha}(x)$.

Therefore,
\[
\sum_{j,k,\alpha,\beta}
 |R_{j\overline{k}\alpha\overline{\beta}}(x)|^2
 = |\sqrt{-1}\Theta_h(E)(x)|_{h,g_{\mathbf{P}}}^2
 \le C^2.
\]

For any
\[
u = \sum_{j,\alpha} u^{j\alpha}
     \left(\frac{\partial}{\partial w_j}\right)\otimes e_\alpha
   \in (T^{1,0}_{X^*}\otimes E)_x,
\]
we have 
\[
\begin{split}
\left| \sum _{j, k, \alpha, \beta} R_{j\overline k \alpha \overline \beta} (x) u^{j\alpha} 
\overline {u^{k\beta}}\right|^2
&\leq 
\left(\sum _{k, \beta} \left| \sum _{j, \alpha}R_{j\overline k\alpha\overline \beta} (x) 
u^{j\alpha} \right|^2\right)
\left(\sum _{k, \beta} |\overline {u^{k\beta}}|^2\right)\\ 
&\leq 
\left(\sum _{k, \beta} \left( \sum _{j, \alpha} \left|
R_{j\overline k\alpha\overline \beta} (x)\right|^2\right)
\left(\sum _{j, \alpha} |u^{j\alpha}|^2\right) \right)
\left(\sum _{k, \beta} |\overline {u^{k\beta}}|^2\right)\\
&= |u|^4_{h,\omega_P} \cdot \sum _{j, k, \alpha, \beta} \left|R_{j \overline k 
\alpha\overline \beta}(x)\right|^2
\\ 
&\leq |u|^4_{h, \omega_P} \cdot C^2  
\end{split} 
\] 
by using the Cauchy--Schwarz inequality twice. 
This gives the desired inequality 
\[
-C |u|^2_{h, \omega_P} \leq \sum _{j, k, \alpha, \beta} R_{j\overline k\alpha 
\overline \beta}(x) u^{j\alpha} \overline{u^{k\beta}}\leq C|u|^2_{h, \omega_P}. 
\]
and completes the proof.
\end{proof}

\begin{defn}[Twisted metric]\label{x-def6.6}
Let $E$ be a holomorphic vector bundle on
$(\Delta^*)^l\times \Delta^{n-l}$ with Hermitian metric $h$. 
For $\bm a=(a_1,\dots,a_l)\in\mathbb{R}^l$ and $N\in\mathbb{R}$, set 
\[
\chi(\bm a,N)
 := -\sum_{j=1}^l a_j \log|z_j|^2
    - N\!\left(
         \sum_{j=1}^l \log(-\log|z_j|^2)
       + \sum_{k=l+1}^n \log(1-|z_k|^2)
       \right).
\]

Define the twisted metric
\[
h(\bm a,N)
 := h\, e^{-\chi(\bm a,N)}
 = h \cdot 
   \prod_{j=1}^l |z_j|^{2a_j}(-\log|z_j|^2)^N
   \prod_{k=l+1}^n (1-|z_k|^2)^N.
\]

Then
\[
\sqrt{-1}\Theta_{h(\bm a,N)}(E)
 = \sqrt{-1}\Theta_h(E)
   + N\,\omega_P \otimes \Id_E.
\]
\end{defn}

By Lemma \ref{x-lem6.5} and Definition \ref{x-def6.6}, we have: 

\begin{cor}\label{x-cor6.7}
Let $(E,h)$ be acceptable on $X^*=(\Delta^*)^l\times\Delta^{n-l}$.
Then there exists $N_0$ such that for all $\bm a\in\mathbb{R}^l$ and all
$N\ge N_0$,
\[
(E,h(\bm a,N)) \ \text{is Nakano semipositive},\qquad
(E,h(\bm a,-N)) \ \text{is Nakano seminegative}.
\]
In particular, $h(\bm a,-N)$ is Griffiths seminegative for $N\ge N_0$.
If $N>N_0$ then 
\[
(E,h(\bm a,N)) \ \text{is Nakano positive},\qquad
(E,h(\bm a,-N)) \ \text{is Nakano negative}.
\]
\end{cor}

\begin{proof}[Proof of Corollary \ref{x-cor6.7}]
Since $h(\bm a,N)=h\,e^{-\chi(\bm a,N)}$, we have
\[
\sqrt{-1}\Theta_{h(\bm a,N)}(E)
 = \sqrt{-1}\Theta_h(E) + N\,\omega_P\otimes\Id_E.
\]
The claim follows immediately from Lemma~\ref{x-lem6.5}.
\end{proof}

\begin{lem}\label{x-lem6.8}
Let $(E,h)$ be acceptable on $(\Delta^*)^l\times\Delta^{n-l}$.
For any $\bm b\in\mathbb{R}^l$, set
\[
(E^\dagger,h^\dagger) := (E, h(\bm b,0)).
\]
Then
\[
\sqrt{-1}\Theta_h(E)=\sqrt{-1}\Theta_{h^\dagger}(E^\dagger),
\qquad
{}_{\bm a}E = {}_{\bm a-\bm b}E^\dagger.
\]
\end{lem}

\begin{proof}[Proof of Lemma \ref{x-lem6.8}]
Since $\partial\overline{\partial}\chi(\bm b,0)=0$ on
$(\Delta^*)^l\times\Delta^{n-l}$, 
we have 
\[
\begin{split}
\sqrt{-1} \Theta_{h^\dag} (E^\dag) &=\sqrt{-1} \Theta_h (E)+
\sqrt{-1} \partial 
\overline \partial \chi (\bm b, 0)\otimes \Id_E \\ 
&=\sqrt{-1} \Theta_h(E)
\end{split}
\] 
The identity of the prolongations is immediate from the definition.
\end{proof}

The following lemma is very well known and it plays a crucial role 
in the theory of acceptable bundles through Corollary \ref{x-cor6.7}. 

\begin{lem}\label{x-lem6.9}
Let $(E,h)$ be a vector bundle on a complex manifold $X$ with
$\sqrt{-1}\Theta_h(E)$ Griffiths seminegative.
Then for every holomorphic section $s$ of $E$, the function $\log|s|^2_h$ is
plurisubharmonic on $X$.
\end{lem}

\begin{proof}[Proof of Lemma \ref{x-lem6.9}]
Let $\{\bullet, \bullet\}_h$ denote the sesquilinear pairing
\[
C^\infty(X, \wedge^p T^\vee_X \otimes E) \times 
C^\infty(X, \wedge^q T^\vee_X \otimes E) \to 
C^\infty(X, \wedge^{p+q} T^\vee_X \otimes \mathbb{C})
\]
induced by the Hermitian metric $h$.

More precisely, let $\Omega$ be an open subset of $X$, 
and assume that $E|_\Omega$ is trivialized 
as $\Omega\times \mathbb C^r$ by a $C^\infty$ frame $\{e_\lambda\}$. 
Then for any sections
\[
u = \sum_\lambda u_\lambda \otimes e_\lambda,
\quad 
v = \sum_\mu v_\mu \otimes e_\mu,
\]
we have
\[
\{u, v\}_h = \sum_{\lambda, \mu} u_\lambda \wedge 
\overline v_\mu \cdot h(e_\lambda, e_\mu).
\]

Let $D_h = D'_h + \overline{\partial}$ denote the Chern 
connection associated with $(E, h)$. 
We may assume that $s\not\equiv 0$. 
Outside the zero set of $s$, we have 
\[
\begin{split}
\sqrt{-1}\partial \overline\partial \log |s|^2_h&= \sqrt{-1} \frac{\{D'_hs, D'_hs\}_h}{|s|^2_h} 
-\sqrt{-1} \frac{\{D'_hs, s\}_h\wedge \{s, D'_hs\}_h} {|s|^4_h} 
-\frac{\{\sqrt{-1} \Theta_h(E)s, s\}_h}{|s|^2_h} \\ 
&\geq -\frac{\{\sqrt{-1} \Theta_h(E)s, s\}_h}{|s|^2_h}\geq 0. 
\end{split}
\] 
We note that the first inequality is due to Cauchy--Schwarz inequality and 
the second one holds since $\sqrt{-1} \Theta_h (E)$ is Griffiths seminegative. 
Thus we have 
\[
\sqrt{-1} \partial \overline{\partial} \log |s|^2_h \geq 0
\]
outside the zero set of $s$. That is, 
$\log |s|^2_h$ is subharmonic on $X \setminus \{s = 0\}$. 

Moreover, since $\log |s|^2_h$ is 
locally bounded from above, it extends to a subharmonic function on all of $X$ (see, for example, 
\cite[(3.3.41) Theorem]{noguchi-ochiai} or  
\cite[Chapter~I, (5.24) Theorem]{demailly}). 

This completes the proof of Lemma \ref{x-lem6.9}.
\end{proof}

We end this section with a very important remark. 

\begin{rem}\label{x-rem6.10}
In \cite[21.2.~Twist of the metric of an acceptable bundle]{mochizuki4}, Mochizuki sets 
$\tau(\bm a,N)
 :=\chi(\bm a,-N)$ and defines $h_{\bm a,N}
 := h\,e^{-\tau(\bm a,N)}$.
Thus $h(\bm a,N)=h_{\bm a,-N}$ in our notation.
If $N$ is sufficiently large, Corollary~\ref{x-cor6.7} shows that
$(E,h(\bm a,N))$ is Nakano positive and
$(E,h(\bm a,-N))$ is Griffiths negative.
In Mochizuki's notation, the roles of $N$ and $-N$ are reversed.
We find our convention more natural, and therefore adopt $\chi(\bm a,N)$
in this paper.
\end{rem}

\section{Some preliminary estimates}\label{x-sec7}

In this section, we collect several preliminary estimates for acceptable 
bundles on a partially punctured polydisk.  
Although all the results in this section can be found in \cite[21.2]{mochizuki4}, 
we present them here in detail, since we adopt a different convention 
(see Remark~\ref{x-rem6.10}).

\medskip

We set
\[
X := \Delta^n = \{ (z_1,\ldots,z_n) \in \mathbb{C}^n \mid |z_i| < 1 \text{ for all } i \}
\]
and
\[
D := \sum_{i=1}^l D_i,
\]
where $D_i := \{ z_i = 0 \}$ for each $i$.  
We put $X^* := X \setminus D$.  
Then clearly
\[
X^* = (\Delta^*)^l \times \Delta^{\,n-l}.
\]

For $i = 1,\ldots,n$, let 
\[
\pi_i \colon X^* \to D_i
\]
denote the natural projection.  
We set
\[
D_i^\circ := D_i \setminus \bigcup_{\substack{j\neq i \\ j \leq l}} D_j.
\]
For any point $P \in D_i^\circ$, we see that 
\[
\pi_i^{-1}(P) \simeq 
\begin{cases}
\Delta^*, & 1 \leq i \leq l,\\[2pt]
\Delta,   & l+1 \leq i \leq n.
\end{cases}
\]

For $0 < R \leq 1$, we define
\[
X(R) := \{ (z_1,\ldots,z_n) \in X \mid |z_i| < R \text{ for all } i \},
\]
and set $X^*(R) := X(R) \cap X^*$.

\medskip

As a direct consequence of Lemma~\ref{x-lem6.9}, we obtain the following corollary.

\begin{cor}[{\cite[Corollary~21.2.5]{mochizuki4}}]\label{x-cor7.1}
Let $(E,h)$ be an acceptable vector bundle on $X^*$.  
Assume that $(E,h(0,-N_0))$ is Griffiths seminegative.  
Let $F$ be a holomorphic section of $E$ on $X^*(R)$ such that 
\[
\|F|_{X^*(R)}\|_{h(\bm{a},-N)} < \infty
\]
for some $0<R\leq 1$.  
Here
\[
\|F|_{X^*(R)}\|_{h(\bm{a},-N)}^2
:= \int_{X^*(R)} |F|_{h(\bm{a},-N)}^2 \, \dvol_{X-D},
\]
where $\dvol_{X-D}$ is the volume form on 
$X^*=(\Delta^*)^l\times \Delta^{n-l}$ with respect to the Poincaré metric $\omega_P$, 
that is, 
\[
\dvol _{X-D} =\frac{\omega^n_P}{n!}.
\]

Then for every $1 \le j \le l$ and every $P \in D_j^\circ$, we have
\[
\int_{\pi_j^{-1}(P)\cap X^*(R')}
\left| F|_{\pi_j^{-1}(P)\cap X^*(R')} \right|^2_{h(\bm{a},-M)}
\, \dvol_{\pi_j^{-1}(P)} < \infty,
\]
for any $0 < R' < R$ and any $M \ge \max\{N_0, N\}$,  
where $\dvol_{\pi_j^{-1}(P)}$ is the volume form induced by the restriction
$\omega_P|_{\pi_j^{-1}(P)}$.

More precisely, there exists a constant $C>0$ such that 
\[
\int_{\pi_j^{-1}(P)\cap X^*(R')}
\left| F|_{\pi_j^{-1}(P)\cap X^*(R')} \right|^2_{h(\bm{a},-M)}
\, \dvol_{\pi_j^{-1}(P)} < C\| F|_{X^*(R)}\|^2_{h(\bm a, -N)}<\infty. 
\]

\end{cor}

\begin{proof}[Proof of Corollary~\ref{x-cor7.1}]
Since $M \ge N_0$, Lemma~\ref{x-lem6.9} implies that
\[
|F|_{h(\bm{a},-M)}^2
= \exp(\log |F|_{h(\bm{a},-M)}^2)
\]
is plurisubharmonic.  
Hence for any complex submanifold $V$ of $X^*(R)$, 
the restriction $|F|_{h(\bm{a},-M)}^2|_V$ is subharmonic.

Let $U$ be a small ball centered at $P$ in $D_j^\circ$.  
Then $\pi_j^{-1}(U) \simeq \pi_j^{-1}(P) \times U$.  
Let $\dvol_U$ be the Euclidean volume form on $U$.  
There exists a constant $C_1>0$ such that
\begin{equation}\label{x-eq7.1}
\dvol_U \cdot \dvol_{\pi_j^{-1}(P)}
\le C_1 \, \dvol_{X-D}
\quad \text{on } \pi_j^{-1}(U) \cap X^*(R').
\end{equation}

For $Q \in \pi_j^{-1}(P) \cap X^*(R')$, the function  
$\left|F|_{\{Q\}\times U}\right|^2_{h(\bm{a},-M)}$ is plurisubharmonic, hence
\begin{equation}\label{x-eq7.2}
|F|_{h(\bm{a},-M)}^2(Q,P)
\le \frac{1}{\mathrm{Vol}(U)}
\int_{\{Q\}\times U}
\left|F|_{\{Q\}\times U}\right|^2_{h(\bm{a},-M)} \, \dvol_U
\end{equation}
by the mean value inequality, where
\[
\mathrm{Vol}(U) := \int_U 1 \, \dvol_U < \infty.
\]

Since $M \ge N$, there exists a constant $C_2>0$ such that
\begin{equation}\label{x-eq7.3}
|F|_{h(\bm{a},-M)}^2 \le C_2 |F|_{h(\bm{a},-N)}^2
\quad \text{on } X^*(R').
\end{equation}

Combining \eqref{x-eq7.1}, \eqref{x-eq7.2}, and \eqref{x-eq7.3}, we obtain
\[
\begin{aligned}
&\int_{\pi_j^{-1}(P)\cap X^*(R')}
\left|F|_{\pi_j^{-1}(P)\cap X^*(R')} \right|_{h(\bm{a},-M)}^2
\, \dvol_{\pi_j^{-1}(P)} \\
&\le \frac{1}{\mathrm{Vol}(U)}
\int_{(\pi_j^{-1}(P)\times U)\cap X^*(R')}
|F|_{h(\bm{a},-M)}^2 \, \dvol_U \, \dvol_{\pi_j^{-1}(P)} \\
&\le \frac{C_1}{\mathrm{Vol}(U)}
\int_{(\pi_j^{-1}(P)\times U)\cap X^*(R')}
|F|_{h(\bm{a},-M)}^2 \, \dvol_{X-D} \\
&\le \frac{C_1 C_2}{\mathrm{Vol}(U)}
\|F|_{X^*(R')}\|_{h(\bm{a},-N)}^2
< \infty.
\end{aligned}
\] 
More precisely, the first inequality follows from the mean value inequality \eqref{x-eq7.2}, 
the second one follows from \eqref{x-eq7.1}, the third one is due to \eqref{x-eq7.3}, 
and the final one follows from the assumption. 
This completes the proof. 
\end{proof}

The following lemma is also a direct consequence of the mean value inequality 
for subharmonic functions.

\begin{lem}[{\cite[Lemma~21.2.6]{mochizuki4}}]\label{x-lem7.2} 
Let $(E,h)$ be an acceptable vector bundle on $X^* = \Delta^*$.  
Let $f$ be a holomorphic section of $E$ on
\[
\Delta^*(R) := \{ z \in \mathbb{C} \mid 0 < |z| < R \}
\]
for some $0 < R \le 1$, and assume that
\[
\| f|_{\Delta^*(R)} \|_{h(b,-M_0)} < \infty
\]
for some $b, M_0 \in \mathbb{R}$.  
Suppose that $(E,h(0,-N_0))$ is Griffiths seminegative.  
Let $M \ge \max\{N_0, M_0 + 2\}$.  
Then
\[
|f(z)|_h^2 
\le 
B \cdot 
\| f|_{\Delta^*(R)} \|_{h(b,-M_0)}^2 
\, |z|^{-2b} \bigl(-\log |z|\bigr)^{M}
\]
holds on $X^*(R/5) = \Delta^*(R/5)$, where $B>0$ is independent of $f$.
\end{lem}

\begin{proof}[Proof of Lemma~\ref{x-lem7.2}]
Since $M \ge N_0$, Lemma~\ref{x-lem6.9} implies that 
$\log |f(z)|_{h(b,-M)}$ is subharmonic on $\Delta^*(R)$.  
Let $\dvol$ denote the Euclidean volume form, and 
let $\dvol_{\omega_P}$ be the volume form associated to the 
Poincaré metric.  
For $0 < |z| \le R/5 \le 1/5$, we have
\[
\begin{aligned}
\log |f(z)|_{h(b,-M)}
&\le \frac{4}{\pi |z|^2}
\int_{|w-z|\le |z|/2} \log |f(w)|_{h(b,-M)}^2 \, \dvol  
\\
&\le 
\log \left(
\frac{4}{\pi |z|^2}
\int_{|w-z|\le |z|/2} |f(w)|_{h(b,-M)}^2 \, \dvol
\right)
\\
&\le 
\log \left(
\frac{9}{\pi}
\int_{|w-z|\le |z|/2} 
\frac{|f(w)|_{h(b,-M)}^2}{|w|^2}
\, \dvol
\right)
\\
&\le
\log \left(
\frac{9}{\pi}
\int_{|w-z|\le |z|/2}
|f(w)|_{h(b,-M_0)}^2 
\, \dvol_{\omega_P}
\right)
\\
&\le 
\log \left(
\frac{9}{\pi}
\| f|_{\Delta^*(R)} \|_{h(b,-M_0)}^2
\right).
\end{aligned}
\]
The first inequality is the mean value inequality;  
the second follows from Jensen's inequality;  
the third uses $|w|\le \tfrac{3}{2}|z|$;  
the fourth follows from $M \ge M_0 + 2$ and $\log|w| < -1$ for 
$|w|\le \tfrac{3}{2}|z| \le \tfrac{3}{10} < e^{-1}$.  
This proves the desired estimate.
\end{proof}

Although the following lemma is elementary, 
it plays a crucial role in this paper.

\begin{lem}[{\cite[Lemma~21.2.7]{mochizuki4}}]\label{x-lem7.3}
Let $(E,h)$ be an acceptable vector bundle on $X^* = \Delta^*$.  
Let $f$ be a holomorphic section of $E$ such that
\[
|f|_h = O\!\left( \frac{1}{|z|^{a+\varepsilon}} \right)
\]
for every $\varepsilon > 0$ on $\Delta^*(R)$ for some $0 < R \le 1$.  
Let $N_0$ be such that $(E,h(0,-N_0))$ is Griffiths seminegative, 
and let $M \ge N_0$.  
Define
\[
H(z) := |f(z)|_h^2 \, |z|^{2a} \, \bigl(-\log|z|\bigr)^{-M}.
\]
Then $H(z)$ is bounded near the origin.  
More precisely,
\[
\max_{|z|\le R'} H(z)
= 
\max_{|z|=R'} H(z)
\]
for every $0 < R' < R$.
\end{lem}

\begin{proof}[Proof of Lemma~\ref{x-lem7.3}]
For $\varepsilon > 0$, set 
\[
H_\varepsilon(z) := H(z)\, |z|^{2\varepsilon}.
\]
By Lemma~\ref{x-lem6.9}, $\log H_\varepsilon(z)$ is subharmonic on 
$\Delta^*(R)$.  
The assumption on $f$ implies 
\[
\lim_{z\to 0} \log H_\varepsilon(z) = -\infty.
\]
Hence $\log H_\varepsilon$ extends as a subharmonic function to $\Delta(R)$
(see \cite[(3.3.25) Theorem]{noguchi-ochiai}).  
Therefore,
\begin{equation}\label{x-eq7.4}
\max_{|z|\le R'} H_\varepsilon(z)
=
\max_{|z|=R'} H_\varepsilon(z).
\end{equation}

Since $H(z)$ is continuous on $\{ |z|=R'\}$ and 
$H_{\varepsilon_1}(z) \le H_{\varepsilon_2}(z)$ holds for 
$0 \le \varepsilon_2 \le \varepsilon_1 \le 1$,  
letting $\varepsilon\to 0$ in \eqref{x-eq7.4} yields the boundedness of $H(z)$ 
on $\{|z|\le R'\}$.  
This completes the proof.
\end{proof}

Proposition~\ref{x-prop7.4} is a direct consequence of 
Lemma~\ref{x-lem7.3}, and it will play a crucial role in the following sections.

\begin{prop}[{\cite[Proposition~21.2.8]{mochizuki4}}]\label{x-prop7.4}
Let $(E,h)$ be an acceptable vector bundle 
on $X^* = X \setminus D=(\Delta^*)^l\times \Delta^{n-l}$.  
Let $F$ be a holomorphic section of $E$ on $X^*(R)$ for some $0 < R \le 1$.  
Assume that there exist real numbers $a_i$ {\em{(}}$1 \le i \le l${\em{)}} such that:

\begin{itemize}
\item 
For every $\varepsilon > 0$, every $1 \le i \le l$, and every 
$P \in D_i^\circ$, we have
\[
\left|F|_{\pi_i^{-1}(P)}\right|_h
= O\!\left(\frac{1}{|z_i|^{a_i +\varepsilon}}\right).
\]
\end{itemize}

Let $N_0$ be such that $(E, h(0,-N_0))$ is Griffiths seminegative, 
and let $M \ge N_0$.  
Fix any real number $0 < R' < R$.  
Then there exists a constant $B > 0$, independent of $F$, such that
\[
|F|_h^2
\le 
B \cdot 
\prod_{j=1}^l \left( |z_j|^{-2a_j} \, (-\log |z_j|)^M \right)
\cdot 
\max_{\substack{|z_j| = R' \\ 1 \le j \le l}} |F|_h^2
\quad \text{on } X^*(R').
\]
\end{prop}

\begin{proof}[Proof of Proposition~\ref{x-prop7.4}]
Set
\[
H(z_1,\ldots,z_l)
:= 
|F|_h^2 \cdot 
\prod_{j=1}^l 
\left( |z_j|^{2a_j} \, (-\log |z_j|)^{-M} \right).
\]
Applying Lemma~\ref{x-lem7.3} to each coordinate $z_i$ successively, 
we obtain the desired inequality.  
This completes the proof.
\end{proof}

Similarly, we obtain the following:

\begin{cor}[{\cite[Corollary~21.2.9]{mochizuki4}}]\label{x-cor7.5}
Let $(E,h)$ be an acceptable vector bundle on $X^*=(\Delta^*)^l
\times \Delta^{n-l}$.  
Suppose that $F$ is a holomorphic section of $E$ on $X^*(R)$ 
for some $0 < R \le 1$, and that
\[
\| F|_{X^*(R)} \|_{h(\bm{a},-M_0)} < \infty.
\]
Let $N_0$ be such that $(E,h(0,-N_0))$ is Griffiths seminegative, 
and let $M \ge \max\{ N_0, M_0 + 2 \}$.  
Then, for any $0 < R' < R$, we have on $X^*(R')$:
\[
|F|_h^2 
\le 
B \cdot 
\prod_{j=1}^l \left( |z_j|^{-2a_j} \, (-\log |z_j|)^M \right)
\cdot 
\max_{\substack{|z_j| = R' \\ 1 \le j \le l}} |F|_h^2,
\]
where $B > 0$ is independent of $F$.  
In particular, $F \in {}_{\bm{a}} E$.
\end{cor}

\begin{proof}[Proof of Corollary~\ref{x-cor7.5}]
The desired estimate follows directly from 
Corollary~\ref{x-cor7.1}, Lemma~\ref{x-lem7.2}, 
and Proposition~\ref{x-prop7.4}.  
This completes the proof.
\end{proof}

In the following sections, we will repeatedly use 
Proposition~\ref{x-prop7.4} and Corollary~\ref{x-cor7.5}.

\section{$L^2$ extension theorem of Ohsawa--Takegoshi type}\label{x-sec8}

In this paper, we use an $L^2$ extension theorem of
Ohsawa--Takegoshi type as a black box. 
We remark that Mochizuki does not rely on the
Ohsawa--Takegoshi $L^2$ extension theorem; instead,
he develops the theory within the framework of
Andreotti--Visentini
(see \cite[21.1.~Some general results on vector bundles on
K\"ahler manifolds]{mochizuki4},
as well as \cite{av} and \cite{cornalba-griffiths}). 
The theorem stated below is a very special case of
\cite[Theorem]{ohsawa} and
\cite[Corollary~3.13]{guan-zhou}.
Since optimal constants are not needed for our purposes,
we restrict ourselves to this weaker formulation.

\begin{thm}[{see \cite[Theorem]{ohsawa} and \cite[Corollary~3.13]{guan-zhou}}]
\label{x-thm8.1}
Let $V$ be a bounded Stein open subset of $\mathbb C^n$ and 
let $(\mathcal E, h)$ be a Nakano semipositive vector bundle 
over $V$. 
Let $\varphi$ be any smooth plurisubharmonic function on $V$ and 
let $s_1, \ldots, s_m$ be linear functions such that 
\[
W:=\{x\in V\mid s_1(x)=\cdots =s_m(x)=0\}
\]
is a closed complex submanifold of codimension $m$. 
We put 
\[
c_k=(\sqrt{-1})^{k^2}
\]
for any positive integer $k$. 
Then, given a holomorphic $\mathcal E$-valued $(n-m)$-form $g$ on $W$ with 
\[
\int_W e^{-\varphi} c_{n-m} \{g, g\}_h < \infty, 
\]
for any $\varepsilon > 0$, 
there exists a holomorphic $\mathcal E$-valued $n$-form 
$G_\varepsilon$ on $V$ which coincides with 
\[
g \wedge ds_1 \wedge \cdots \wedge ds_m
\]
on $W$ and satisfies 
\[
\int_V e^{-\varphi} 
\left(1+\sum_{i=1}^m |s_i|^2\right)^{-m-\varepsilon} 
c_n \{G_\varepsilon, G_\varepsilon\}_h 
\leq \frac{C}{\varepsilon} 
\int_W e^{-\varphi} c_{n-m} \{g, g\}_h < \infty, 
\]
where $C$ is a positive constant independent of $g$. 
\end{thm}

\begin{proof}[Proof of Theorem~\ref{x-thm8.1}]
This theorem is a direct consequence of 
\cite[Theorem]{ohsawa}. 
Readers interested in optimal constants are referred to 
\cite[Corollaries~3.13 and~3.14]{guan-zhou}. 
The above non-optimal version suffices for our purposes.
\end{proof}

Since Theorem~\ref{x-thm8.1} is not a standard formulation of the 
Ohsawa--Takegoshi $L^2$ extension theorem, 
we include below a more familiar version for the reader's convenience. 
Of course, Theorem~\ref{x-thm8.2} is a special case of 
Theorem~\ref{x-thm8.1}.

\begin{thm}[Ohsawa--Takegoshi $L^2$ extension theorem]\label{x-thm8.2}
Let $V\subset\mathbb C^n$ be a bounded Stein open set, and let
$\mathcal E$ be a holomorphic vector bundle over $V$
equipped with a smooth Hermitian metric $h$ that is Nakano semipositive.
Let $\varphi$ be a smooth plurisubharmonic function on $V$.
Let $s$ be a nonzero linear function on $\mathbb C^n$, and set
\[
H := V \cap \{s = 0\}.
\]
Let $f$ be a holomorphic section of $\mathcal E|_H$ such that
\[
\int_H |f|_h^2 e^{-\varphi}\, d\lambda_{n-1} < \infty,
\]
where $d\lambda_{n-1}$ denotes the Lebesgue measure on
$\mathbb C^{n-1} = \{s = 0\}$.
Then there exists a holomorphic section $F$ of $\mathcal E$ on $V$
satisfying $F|_H = f$ and
\[
\int_V |F|_h^2 e^{-\varphi}\, d\lambda_n
\le C'' \int_H |f|_h^2 e^{-\varphi}\, d\lambda_{n-1},
\]
where $d\lambda_n$ denotes the Lebesgue measure on $\mathbb C^n$ 
and $C''>0$ is a constant independent of $f$ 
\end{thm}

\begin{proof}[Proof of Theorem~\ref{x-thm8.2}]
The original Ohsawa--Takegoshi $L^2$ extension theorem is formulated 
for holomorphic functions. 
However, the same argument applies to holomorphic sections of 
Nakano semipositive vector bundles. 
Indeed, Theorem~\ref{x-thm8.2} follows from the standard proof 
given in \cite{ohsawa-takegoshi} and \cite[2 Proof of 
Theorem 0.2]{ohsawaIII}, 
combined with a variant of Kodaira--Nakano's vanishing 
theorem (see \cite[Theorem 1.7]{ohsawaIII} and 
\cite[Theorem 5]{ohsawa-v}). 
We omit the details.
\end{proof}

\begin{cor}\label{x-cor8.3}
Let $V$ be a bounded Stein open subset of $\mathbb C^n$ and 
let $(\mathcal E, h)$ be a Nakano semipositive vector bundle 
over $V$. 
Let $\varphi$ be any smooth plurisubharmonic function on $V$. 
Let $(z, w_2, \ldots, w_n)$ be a coordinate system of $\mathbb C^n$. 
We put 
\[
W:=\{x\in V\mid w_2(x)=\cdots =w_n(x)=0\}. 
\] 
Let $f(z)$ be a holomorphic section of $\mathcal E|_W$ on $W$ such that 
\[
\int _W |f|^2_h e^{-\varphi} \frac{\sqrt{-1}}{2} dz\wedge d\overline{z}<\infty. 
\] 
Then there exists a holomorphic section $F$ of $\mathcal E$ on $V$ such that 
\[
\int _V |F|^2_h e^{-\varphi} d\lambda_n \leq C' 
\int _W |f|^2_h e^{-\varphi} \frac{\sqrt{-1}}{2} dz\wedge d\overline{z}<\infty, 
\] 
where $d\lambda_n$ denotes the Lebesgue measure of $\mathbb C^n$ and 
$C'$ is a positive number which does not depend on $f$. 
\end{cor}

\begin{proof}[Proof of Corollary~\ref{x-cor8.3}]
The corollary is an immediate consequence of 
Theorem~\ref{x-thm8.1}. 
For the reader's convenience, we briefly indicate the argument.

Set $g := f\,dz$. 
Then $g$ is a holomorphic $\mathcal E|_W$-valued $1$-form on $W$ satisfying
\[
\int_W e^{-\varphi} c_1 \{g, g\}_h < \infty.
\]
Applying Theorem~\ref{x-thm8.1}, we obtain a holomorphic 
$\mathcal E$-valued $n$-form $G$ on $V$ of the form
\[
G = F\, dz \wedge dw_2 \wedge \cdots \wedge dw_n,
\]
such that $F|_W = f$ and
\[
\int_V e^{-\varphi} c_n \{G, G\}_h
\le C^\sharp \int_W e^{-\varphi} c_1 \{g, g\}_h < \infty,
\]
where $C^\sharp>0$ is independent of $f$.
Since $c_n \{G, G\}_h$ is a constant multiple of 
$|F|_h^2 d\lambda_n$, the desired estimate follows.
\end{proof}

\begin{rem}\label{x-rem8.4}
Corollary~\ref{x-cor8.3} can alternatively be obtained by applying
Theorem~\ref{x-thm8.2} inductively along the flag
\[
W = W_2 \subset W_3 \subset \cdots \subset W_n \subset V,
\]
where
\[
W_i := \{ x \in V \mid w_j(x) = 0 \text{ for } i \le j \le n \}.
\]
\end{rem}

In the subsequent sections, we will frequently use the following
form of the Ohsawa--Takegoshi $L^2$ extension theorem, which is a direct
consequence of Corollary~\ref{x-cor8.3}.

\begin{prop}\label{x-prop8.5}
Let $0 < R < 1$. Define
\[
X^*(R) := \left\{ (z_1, \ldots, z_n) \in \mathbb{C}^n \,\middle|\, 
\begin{array}{l}
0 < |z_i| < R \quad \text{for } 1 \leq i \leq l, \\
|z_i| < R \quad \text{for } l+1 \leq i \leq n 
\end{array}
\right\}.
\]
Then $X^*(R)$ is a bounded Stein open subset of $\mathbb{C}^n$. 
Let $(\mathcal{E}, h)$ be a Nakano semipositive vector bundle over $X^*(R)$. 
We define the new coordinates as follows:
\[
z := z_1, \qquad 
w_i := \begin{cases}
z_i - z_1 & \text{for } 2 \leq i \leq l, \\
z_i       & \text{for } l+1 \leq i \leq n.
\end{cases}
\] 
Define the submanifold
\[
Y^*(R) := \left\{ (z_1, \ldots, z_n) 
\in X^*(R) \,\middle|\, w_2 = \cdots = w_n = 0 \right\}.
\] 
Set the weight functions
\[
\psi := \frac{1}{l} \sum_{i=1}^l \log |z_i|^2, \qquad
\phi := -\left(1 - \frac{1}{l}\right) \sum_{i=1}^l \log |z_i|^2, 
\quad \text{and}\quad 
\phi_{\bm a} := -\sum_{i=1}^l a_i\log |z_i|^2.
\]

Let $f$ be a holomorphic section of $\mathcal{E}|_{Y^*(R)}$ satisfying
\[
\int_{Y^*(R)} |f|^2_h e^{-\psi-\phi_{\bm a}}
 \cdot \frac{\sqrt{-1}}{2} \, dz \wedge d\overline{z} < \infty.
\]
Then there exists a holomorphic section $F$ of $\mathcal{E}$ on $X^*(R)$ such that
\begin{equation}\label{x-eq8.1}
F|_{Y^*(R)} = f, \qquad 
\int_{X^*(R)} |F|^2_h e^{-\psi-\phi_{\bm a}} \, d\lambda_n < \infty.
\end{equation}
Therefore, we also have
\begin{equation}\label{x-eq8.2}
\int_{X^*(R)} |F|^2_h e^{-\phi-\phi_{\bm a}} \frac{\omega^n_P}{n!} < \infty.
\end{equation}
We note that $\phi\equiv 0$ when $l=1$. 
\end{prop}

\begin{proof}[Proof of Proposition~\ref{x-prop8.5}]
Note that $\psi + \phi_{\bm a}$ is a smooth 
plurisubharmonic function on $X^*(R)$
for any $\bm a \in \mathbb R^l$.
Hence, by Corollary~\ref{x-cor8.3}, there exists a holomorphic section
$F$ of $\mathcal E$ on $X^*(R)$ satisfying \eqref{x-eq8.1}.

Moreover, since $0 < R < 1$, there exists a positive constant $C^\dagger$
such that
\[
e^{-\phi-\phi_{\bm a}} \frac{\omega^n_P}{n!}
\le C^\dagger e^{-\psi-\phi_{\bm a}} \, d\lambda_n
\quad \text{on } X^*(R).
\]
Therefore, the integrability condition \eqref{x-eq8.2} follows immediately.
\end{proof}

\section{Acceptable bundles on $\Delta^*$}\label{x-sec9} 

In this section, we briefly recall acceptable bundles on $\Delta^*$ following 
\cite{fujino-fujisawa-ono}. 
We strongly recommend the interested reader to see \cite{fujino-fujisawa-ono}. 

\begin{thm}[{see \cite[Theorem 1.9]{fujino-fujisawa-ono}}]\label{x-thm9.1}
Let $(E, h)$ be an acceptable vector 
bundle on $\Delta^*$ with $\rank E 
= r$. 
Then ${}_a E$ is a holomorphic vector bundle for every $a\in 
\mathbb R$. 
Let $\{v_1, \ldots, v_r\}$ be a local frame of ${}_a E$ near the origin. Define
\[
\gamma({}_a E) := -\frac{1}{2} \liminf_{z \to 0} \frac{\log \det H(h, \bm v)}{\log |z|},
\]
where $H(h, \bm v)$ is the $r \times r$ matrix 
$\left( h(v_i, v_j) \right)$. Then 
$\gamma({}_a E)$ is a well-defined real-valued invariant of ${}_a E$.

Furthermore, if we let
\[
\Par_a(E, h) =: \{b_1, \ldots, b_r\},
\]
then we have
\[
\gamma({}_a E) = -\frac{1}{2} 
\lim_{z \to 0} \frac{\log \det H(h, \bm v)}{\log |z|} = \sum_{i=1}^r b_i.
\] 

Note that if we define 
\[
\{\lambda_1, \ldots, \lambda_k\} 
:= \{ \lambda \in (a-1, a] \mid {}_\lambda E / {}_{<\lambda} E \ne 0 \}
\]
with $\lambda_i \ne \lambda_j$ for $i \ne j$, then
\[
\sum_{i=1}^r b_i = \sum_{i=1}^k \lambda_i 
\dim_{\mathbb{C}} \left( {}_{\lambda_i} E / {}_{<\lambda_i} E \right).
\]
\end{thm}

The following easy lemma will be used in Section \ref{x-sec11}. 

\begin{lem}\label{x-lem9.2}
Let $(E, h)$ be an acceptable vector bundle on $\Delta^*$. 
\begin{itemize}
\item[(i)] Let $\mathcal D$ be a dense subset of $\mathbb R$. 
Then the family $\{\gamma({}_a E)\}_{a\in \mathcal D}$ uniquely determines 
$\{\gamma({}_a E)\}_{a\in \mathbb R}$. 

\item[(ii)] For any $\alpha \in \mathbb R$, 
$\Par_\alpha (E, h)$ is uniquely determined by 
$\{\gamma({}_a E)\}_{a\in \mathbb R}$. 
\end{itemize}
In particular, for any dense subset $\mathcal D \subset \mathbb R$, 
the family $\{\gamma({}_a E)\}_{a\in \mathcal D}$ uniquely determines 
$\Par_\alpha (E, h)$ for all $\alpha \in \mathbb R$. 
\end{lem}

\begin{proof}[Proof of Lemma \ref{x-lem9.2}]
Since $\gamma ({}_a E)$ is right-hand continuous by 
\cite[Lemma 7.10]{fujino-fujisawa-ono}, we can recover 
$\{\gamma({}_a E)\}_{a\in \mathbb R}$ by 
$\{\gamma({}_a E)\}_{a\in \mathcal D}$. 
Thus we have (i). 
We write 
\[
\Par_\alpha(E, h) := \{\underbrace{\lambda_1, \ldots, \lambda_1}_{l_1\text{ times}}, \ldots,
\underbrace{\lambda_k, \ldots, \lambda_k}_{l_k\text{ times}}\}.
\] 
We note that $\gamma ({}_\lambda E)-\gamma({}_{\lambda-\varepsilon} E)\ne 0$ 
for $0<\varepsilon \ll 1$ if and only if 
$\lambda \in \Par_\alpha (E, h)$. 
Moreover, we have 
\[
\gamma ({}_{\lambda_i} E)-\gamma ({}_{\lambda_i-\varepsilon} E)=l_i
\] 
for 
$0<\varepsilon \ll 1$ by 
\cite[Lemma 7.12]{fujino-fujisawa-ono}. 
Hence we obtain (ii). 
\end{proof}

For the details of acceptable bundles on a punctured disk, 
see \cite{fujino-fujisawa-ono}. 

\section{Pull-back and descent revisited}\label{x-sec10}

The behavior of acceptable vector bundles on a punctured disk 
under pull-back by cyclic coverings has already been discussed in 
\cite[Section~11]{fujino-fujisawa-ono}.  
In this section, we revisit this topic from a slightly different perspective.  
Because the literature employs various notational conventions, 
one of our aims here is to clarify the notation that will be used in the following sections.  
Throughout this section, we closely follow 
\cite[21.4.2.~Pull-back and descent]{mochizuki4}.

\begin{defn}\label{x-def10.1}
For any $a,b\in \mathbb R$, we define
\[
\nu(a,b):=\lfloor b-a\rfloor \in \mathbb Z,
\]
that is, $\nu(a,b)$ is the unique integer satisfying
\[
b-1 < \nu(a,b)+a \leq b.
\]
\end{defn}

We now examine the behavior of acceptable vector bundles on a punctured disk 
under pull-back via cyclic coverings.

\medskip

Let $X:=\Delta$ and $X^*:=\Delta^*$.  
Fix a positive integer $c$, and let 
\[
\psi_c \colon X \to X, \qquad \psi_c(z)=z^c,
\]
be the cyclic covering of degree $c$.  
Let $(E,h)$ be an acceptable vector bundle on the target space $X^*$.  
Then its pull-back
\[
(\widetilde E, \widetilde h):=\psi_c^*(E,h)
\]
is again an acceptable vector bundle on the source space $X^*$. 
We sometimes simply write $\psi^{-1}_cE$ to denote $(\widetilde E, \widetilde h)$. 

Let $\bm v=\{v_1,\ldots,v_r\}$ be a frame of 
${}^\diamond\!E={}_0E$ compatible with the parabolic filtration, 
so that $v_i\in {}_{a_i}E\setminus {}_{<a_i}E$ for each $i$, 
where $a_i\in (-1,0]$.  
Define
\[
\widetilde v_i := z^{-\nu(ca_i,b)}\, \psi_c^*(v_i).
\]

\begin{lem}\label{x-lem10.2}
Let $\widetilde {\bm v}=\{\widetilde v_1,\ldots,\widetilde v_r\}$.  
Then $\widetilde {\bm v}$ is a frame of ${}_b\widetilde E$ 
compatible with the parabolic filtration.  
In particular,
\[
\Par({}_b\widetilde E)
=
\{\nu(ca,b)+ca \mid a\in \Par({}^\diamond\!E)\}.
\]
\end{lem}

\begin{proof}[Proof of Lemma \ref{x-lem10.2}]
This is proved in \cite[Lemma~11.2]{fujino-fujisawa-ono}.  
We refer the reader to the proof there for details, 
although the notation used here is slightly different.
\end{proof}

Let $\mu_c := \mathbb{Z}/c\mathbb{Z}$ denote the Galois group of 
$\psi_c \colon X \to X$, and let $g$ be a generator of $\mu_c$.  
The group $\mu_c$ acts on $X$ by multiplication, and this action lifts to 
${}_b \widetilde{E}$. For each $i$, we have
\[
g^*(\widetilde v_i)=\zeta^{-\nu(ca_i,b)} \widetilde v_i,
\]
where $\zeta$ is a primitive $c$-th root of unity.

From now on, assume that $0\le b<1/2$.  
If $c$ is sufficiently large, then:

\begin{itemize}
\item $0\le \nu(ca,b)\le c-1$ for every $a\in \Par({}^\diamond\!E)$, and  
\item the map $\Par({}^\diamond\!E)\to \mathbb Z$, $a\mapsto \nu(ca,b)$, is injective.
\end{itemize}

Let $0$ denote the origin of $X=\Delta$.  
We obtain the following vector space decomposition:
\begin{equation}\label{x-eq10.1}
{}_b \widetilde E|_0 = \bigoplus_{0\le p\le c-1} V_p,
\end{equation}
where
\[
V_p = \langle \widetilde v_j|_0 \mid \nu(ca_j,b)=p \rangle.
\]
Then $g$ acts on $V_p$ by multiplication by $\zeta^{-p}$.

By definition, for each $p$ with $V_p\neq 0$, 
there exists a unique 
\[
\chi(p) \in \Par({}^\diamond\!E)
\quad\text{such that}\quad
p=\nu(c\chi(p),b).
\]
Thus we obtain an injection
\begin{equation}\label{x-eq10.2}
\chi\colon \{0\le p\le c-1 \mid V_p\ne 0\}
\longrightarrow \Par({}^\diamond\!E).
\end{equation}
Set
\[
\varphi(p):=\nu(c\chi(p),b)+c\chi(p)\in \Par({}_b\widetilde E),
\]
giving a map
\begin{equation}\label{x-eq10.3}
\varphi:
\{0\le p\le c-1 \mid V_p\ne 0\}
\longrightarrow
\Par({}_b\widetilde E).
\end{equation}

The decomposition \eqref{x-eq10.1} induces a splitting of the parabolic filtration $F$ of ${}_b\widetilde E$:
\[
F_d({}_b\widetilde E|_0)
=
\bigoplus_{\varphi(p)\le d} V_p.
\]

Conversely, let 
\[
\widetilde{\bm u}=\{\widetilde u_1,\ldots,\widetilde u_r\}
\]
be a $\mu_c$-equivariant frame of ${}_b\widetilde E$, 
so that
\[
g^*\widetilde u_j=\zeta^{-p_j}\widetilde u_j
\quad
(0\le p_j\le c-1).
\]
Then $\widetilde u_j|_0\in V_{p_j}$, and in particular, 
$\widetilde{\bm u}$ is compatible with the filtration $F$.  
Set
\[
u_j:=z^{p_j}\widetilde u_j.
\]
Then each $u_j$ is $\mu_c$-invariant and hence descends to a section of $E$, 
which we also denote by $u_j$.

\begin{lem}\label{x-lem10.3}
The set $\bm u=\{u_1,\ldots,u_r\}$ is a frame of ${}^\diamond\!E$ 
compatible with the parabolic filtration.
\end{lem}

\begin{proof}[Proof of Lemma \ref{x-lem10.3}]
Let $w$ be the coordinate on the target space $X$, so $w=\psi_c(z)=z^c$.  
By definition, for every $\varepsilon>0$, 
there exists a constant $C>0$ such that
\[
|\widetilde u_j|_{\widetilde h}
\le
\frac{C}{|z|^{\varphi(p_j)+\varepsilon}}.
\]
Since
\[
\varphi(p_j)
=
\nu(c\chi(p_j),b)+c\chi(p_j)
=
p_j + c\chi(p_j),
\]
we obtain
\[
|u_j|_h
=
|z^{p_j}\widetilde u_j|_{\widetilde h}
\le
\frac{C}{|w|^{\chi(p_j)+\varepsilon/c}}.
\]
Thus $u_j\in {}_{\chi(p_j)}E\subset {}^\diamond\!E$.  
Because 
\(
\widetilde u_j \in {}_{\varphi(p_j)}\widetilde E
\setminus {}_{<\varphi(p_j)}\widetilde E,
\)
we also have 
\(
u_j\in {}_{\chi(p_j)}E\setminus {}_{<\chi(p_j)}E.
\)
As in the proof of \cite[Lemma~11.3]{fujino-fujisawa-ono}, 
this implies that $\bm u$ is a frame of ${}^\diamond\!E$ 
compatible with the parabolic filtration.
\end{proof}

We end this section with an important remark.

\begin{rem}\label{x-rem10.4}
In \cite[21.4.2]{mochizuki4}, Mochizuki uses the weak norm estimate 
(see \cite[Theorem~21.3.2]{mochizuki4}).  
In contrast, in \cite[Section~11]{fujino-fujisawa-ono}, 
we do not make use of the weak norm estimate 
(see \cite[Theorem~1.13]{fujino-fujisawa-ono}).  
This is because, in \cite{fujino-fujisawa-ono}, the weak norm estimate is proved in 
\cite[Section~13]{fujino-fujisawa-ono}, 
where the argument depends on the results of 
\cite[Section~11]{fujino-fujisawa-ono}.
\end{rem}

\section{Acceptable line bundles on $\Delta^*\times \Delta^{n-1}$}
\label{x-sec11}

In this section, we study an acceptable line 
bundle $(L,h)$ on a partially punctured polydisk 
\[
X^* := \Delta^* \times \Delta^{n-1}.
\]
Our approach is a natural extension of the method developed in
\cite{fujino-fujisawa-ono}, and appears to 
be new and different from that of
Mochizuki.

\medskip

We consider the projection
\[
\pi \colon \Delta^*\times \Delta^{n-1} \to \Delta^{n-1}.
\]

\begin{lem}\label{x-lem11.1}
Let $(L,h)$ be an acceptable line bundle on a partially punctured 
polydisk $X^*=\Delta^*\times \Delta^{n-1}$.
Set $P:=(0,\ldots,0)\in \Delta^{n-1}$.
Assume that
\begin{equation}\label{x-eq11.1}
\alpha \notin \Par_\alpha\!\left(L|_{\pi^{-1}(P)},\, h|_{\pi^{-1}(P)}\right).
\end{equation}
Then ${}_\alpha L$ is a line bundle on $X(R)=\Delta(R)^n$ for some $0<R<1$.
\end{lem}

A more detailed description of a local generator of ${}_\alpha L$ can be found
in the proof of Lemma~\ref{x-lem11.1}.

\begin{proof}[Proof of Lemma~\ref{x-lem11.1}]
Choose a sufficiently large positive integer $N$ such that
\[
h(\alpha,N):=h\cdot e^{-\chi(\alpha,N)}
\]
is Nakano semipositive.
Let $f$ be a generator of ${}_\alpha\!\left(L|_{\pi^{-1}(P)}\right)$.
By assumption \eqref{x-eq11.1}, we have
\[
\int_{\pi^{-1}(P)\cap X(R)} |f|^2_{h(\alpha,N)} e^{-\psi}
\frac{\sqrt{-1}}{2}dz_1\wedge d\overline z_1 < \infty
\]
for any $0<R<1$, where $\psi=\log|z_1|^2$.

By Proposition~\ref{x-prop8.5}, there exists a holomorphic section $F$ of $L$ on
$X^*(R)$ such that $F|_{\pi^{-1}(P)}=f$ and
\[
\|F|_{X^*(R)}\|^2_{h(\alpha, N)}=\int_{X^*(R)} |F|^2_{h(\alpha,N)}\frac{\omega_P^n}{n!} < \infty.
\]
Hence, by Corollary~\ref{x-cor7.5}, we have $F\in {}_\alpha L$.

Set $g:=f^{-1}$.
Applying the same argument to $L^\vee$, we obtain a holomorphic section $G$ of
$L^\vee$ on $X^*(R)$ such that $G|_{\pi^{-1}(P)}=g$ and
$G\in {}_{1-\alpha-\varepsilon}(L^\vee)$ for some $0<\varepsilon\ll1$.
Since $(F\cdot G)|_{\pi^{-1}(P)}\equiv1$, after replacing $G$ by
$\frac{G}{F\cdot G}$ we may assume that $F\cdot G\equiv1$ on $X(R)$.

\begin{claim}\label{x-claim11.2}
The section $F$ is a generator of ${}_\alpha L$, and $G$ is a dual generator of
${}_{1-\alpha-\varepsilon}(L^\vee)$ on $X(R)$.
In particular, both ${}_\alpha L$ and ${}_{1-\alpha-\varepsilon}(L^\vee)$ are line
bundles on $X(R)$, and
\[
({}_\alpha L)^\vee = {}_{1-\alpha-\varepsilon}(L^\vee).
\]
\end{claim}

\begin{proof}[Proof of Claim \ref{x-claim11.2}]
Let $\Phi\in {}_\alpha L$.
Then $\Phi\cdot G\in {}_{1-\varepsilon}\left(\mathcal O_{X^*}\right)=\mathcal O_X$, and hence
$\Phi=(\Phi\cdot G)F$.
This shows that ${}_\alpha L=\mathcal O_X\cdot F$ on $X(R)$.
The statement for $L^\vee$ follows similarly.
\end{proof}

Using the same argument, for any $Q\in \Delta(R)^{n-1}$,
the restrictions $F|_{\pi^{-1}(Q)}$ and $G|_{\pi^{-1}(Q)}$ generate
${}_\alpha(L|_{\pi^{-1}(Q)})$ and
${}_{1-\alpha-\varepsilon}(L^\vee|_{\pi^{-1}(Q)})$, respectively.
In particular,
\[
{}_\alpha(L|_{\pi^{-1}(Q)}) = ({}_\alpha L)|_{\pi^{-1}(Q)}.
\]
This completes the proof of Lemma~\ref{x-lem11.1}.
\end{proof}

Lemma~\ref{x-lem11.3} below is one of the key points of our approach.

\begin{lem}\label{x-lem11.3}
In the setting of Lemma~\ref{x-lem11.1},
$\gamma\!\left({}_\alpha(L|_{\pi^{-1}(Q)})\right)$ is independent of
$Q\in \Delta(R)^{n-1}$.
\end{lem}

\begin{proof}[Proof of Lemma \ref{x-lem11.3}]
By trivializing ${}_\alpha L$ using $F$, we may assume $F\equiv G\equiv1$ on $X(R)$. 
We take a sufficiently large positive real number $N$. 
Then 
\[
\log |F|_{h\cdot e^{-\chi (\alpha, -N)}}
\]
is plurisubharmonic on $X^*(R)$ by Lemma \ref{x-lem6.9}. 
Thus, for any $\alpha'>\alpha$, we see that 
\[
\log |F|_{h\cdot e^{-\chi (\alpha', -N)}}
\] 
is plurisubharmonic on $X(R)$ (see, for example, 
\cite[(3.3.41) Theorem]{noguchi-ochiai} or  
\cite[Chapter~I, (5.24) Theorem]{demailly}).
Similarly, we may assume that 
\[
\log |G|_{h^\vee \cdot e^{-\chi (\beta', -N)}}
\] is also plurisubharmonic on $X(R)$ for any $\beta'>1-\alpha-\varepsilon$. 

We can write $h=|\cdot |^2 e^{-2\varphi_\alpha}$. 
Then we obtain that 
\begin{equation}\label{x-eq11.2}
-2\varphi_\alpha -\chi (\alpha', -N)
\end{equation} 
and 
\begin{equation}\label{x-eq11.3}
2\varphi_\alpha -\chi (\beta', -N)
\end{equation}  
are plurisubharmonic. 
By considering the Lelong number at $Q'=(0, Q)$, we obtain 
\begin{equation}\label{x-eq11.4}
\begin{split}
&\liminf_{z\to Q'} \frac{-2\varphi_\alpha-\chi(\alpha', -N)}{\log |z-Q'|} 
\\ &=
\lim_{r\to +0} 
\frac{1}{r^{2(n-1)}} 
\int _{B(Q', r)} \frac{\sqrt{-1}}{\pi} \partial \overline \partial (-2\varphi_\alpha) 
\wedge \left(\frac{\sqrt{-1}}{2\pi} \sum _{i=1}^n dz_i \wedge d \overline z_i
\right)^{n-1}+2\alpha'
\end{split}
\end{equation}
and 
\begin{equation}\label{x-eq11.5}
\begin{split}
&\liminf_{z\to Q'} \frac{2\varphi_\alpha-\chi(\beta', -N)}{\log |z-Q'|} \\ &=
\lim_{r\to +0} 
\frac{1}{r^{2(n-1)}} 
\int _{B(Q', r)} \frac{\sqrt{-1}}{\pi} \partial \overline \partial (2\varphi_\alpha) 
\wedge \left(\frac{\sqrt{-1}}{2\pi} \sum _{i=1}^n dz_i \wedge d \overline z_i
\right)^{n-1} +2\beta' 
\end{split} 
\end{equation} 
by Lemma \ref{x-lem5.6}. We put 
\[
\Psi (Q'):= \lim_{r\to +0} 
\frac{1}{r^{2(n-1)}} 
\int _{B(Q', r)} \frac{\sqrt{-1}}{\pi} \partial \overline \partial (2\varphi_\alpha) 
\wedge \left(\frac{\sqrt{-1}}{2\pi} \sum _{i=1}^n dz_i \wedge d \overline z_i
\right)^{n-1}
\] 
Then $\Psi$ is an $\mathbb R$-valued function on $V:=\{0\}\times 
\Delta(R)^{n-1}$. 
By Siu's theorem in Definition \ref{x-def5.2} (see, for example, \cite[(13.3) Corollary]{demailly1}), 
\[
\{x\in V\mid \Psi (x)\geq a\} \quad \text{and} \quad 
\{x\in V\mid \Psi (x)\leq b\} 
\] 
are closed analytic subsets of $V$. 
By Lemma \ref{x-lem5.7}, we obtain that 
$\Psi$ is constant on $V$. We put $\Psi :=2A\in \mathbb R$.  
Then, by \eqref{x-eq11.4} and the convexity properties of 
plurisubharmonic functions, that is, Lemma \ref{x-lem5.3} and 
Corollary \ref{x-cor5.5}, we have 
\begin{equation}\label{x-eq11.6}
-2\varphi_\alpha -\chi (\alpha', -N) \leq (-2A+2\alpha') \log \frac{|z-Q'|}{R_0} +M_1, 
\end{equation} 
where $M_1$ is the maximum of $-2\varphi_\alpha -\chi (\alpha', -N)$ on $\overline 
B(Q', R_0)\subset X(R)$. 
Similarly, by \eqref{x-eq11.5} and Corollary \ref{x-cor5.5}, we have 
\begin{equation}\label{x-eq11.7}
2\varphi_\alpha -\chi (\beta', -N) \leq (2A+2\beta') \log \frac{|z-Q'|}{R_0} +M_2, 
\end{equation} 
where $M_2$ is the maximum of $2\varphi_\alpha -\chi (\beta', -N)$ on $\overline B(Q', R_0)
\subset X(R)$.
By \eqref{x-eq11.6}, we obtain 
\begin{equation}\label{x-eq11.8} 
\liminf _{z_1\to 0} \frac{-\varphi_\alpha|_{\pi^{-1}(Q)}}{\log |z_1|}\geq -A. 
\end{equation} 
By \eqref{x-eq11.7}, we have 
\begin{equation}\label{x-eq11.9} 
\liminf _{z_1\to 0} \frac{\varphi_\alpha|_{\pi^{-1}(Q)}}{\log |z_1|}\geq A. 
\end{equation} 
Therefore, by \eqref{x-eq11.8} and \eqref{x-eq11.9}, we have 
\begin{equation}\label{x-eq11.10}
A\leq \liminf_{z_1\to 0} \frac{\varphi_\alpha|_{\pi^{-1}(Q)}}{\log |z_1|}
\leq \limsup _{z_1\to 0} \frac{\varphi_\alpha|_{\pi^{-1}(Q)}}{\log |z_1|}
\leq A. 
\end{equation}
Thus, we obtain 
\begin{equation}\label{x-eq11.11} 
\gamma \left({}_\alpha (L|_{\pi^{-1}(Q)})\right) =\lim_{z_1\to 0} 
\frac{\varphi_\alpha|_{\pi^{-1}(Q)}}{\log |z_1|}=A. 
\end{equation} 
This means that 
\[
\gamma \left({}_\alpha (L|_{\pi^{-1}(Q)})\right)
\] 
is constant with respect to $Q\in \left(\Delta(R)\right)^{n-1}$. 
This is what we wanted. 
\end{proof}

By the above results, we have the following statement. 

\begin{prop}\label{x-prop11.4}
For any $\alpha \in \mathbb R$, ${}_\alpha L$ is a line bundle 
on $X(R)$ for some $0<R<1$. 
Moreover, $\gamma \left({}_\alpha (L|_{\pi^{-1}(Q)})\right)$ is independent of 
$Q\in \Delta(R)^{n-1}$. 
\end{prop}
\begin{proof}[Proof of Proposition \ref{x-prop11.4}]
By Lemma \ref{x-lem11.1}, we may assume that 
\[
\alpha \in \Par_\alpha \left(L|_{\pi^{-1}(P)}, h|_{\pi^{-1}(P)}\right). 
\]
Let $f$ be a generator of ${}_\alpha (L|_{\pi^{-1}(P)})$ on $\pi^{-1}(P)$. 
We take a sufficiently small positive real number $\delta$. 
Then 
\[
\int _{\pi^{-1}(P)\cap X(R)} |f|^2_{h\cdot e^{-\chi (\alpha +\delta, N)}} 
e^{-\psi}\cdot 
\frac{\sqrt{-1}}{2} dz_1\wedge d\overline z_1<\infty,  
\] 
where $\psi=\log |z_1|^2$. 
By Proposition \ref{x-prop8.5} and Corollary \ref{x-cor7.5}, 
we can take a holomorphic section $F$ of $L$ on $X^*(R)$ 
such that $F|_{\pi^{-1}(P)}=f$ and $F\in{}_{\alpha +\delta}L$. 
We may assume that 
$\alpha +\delta\not\in \Par_{\alpha+\delta} \left(L|_{\pi^{-1}(P)}, 
h|_{\pi^{-1}(P)}\right)$. 
By Lemma \ref{x-lem11.3} and \cite[Lemma 13.1]{fujino-fujisawa-ono}, 
we have 
\[
\alpha=\gamma \left({}_{\alpha+\delta} (L|_{\pi^{-1}(P)})\right)
=\gamma \left({}_{\alpha+\delta} (L|_{\pi^{-1}(Q)})\right)
\]
for any $Q\in (\Delta(R))^{n-1}$. 
Thus, by Proposition \ref{x-prop7.4}, 
we have $F\in {}_\alpha L$. Similarly, by 
Proposition \ref{x-prop8.5} and Corollary \ref{x-cor7.5}, we can extend $g:=f^{-1}$ and obtain $G\in 
{}_{1-\alpha -\delta-\varepsilon} (L^\vee)$ on $X^*(R)$ such that 
$G|_{\pi^{-1}(P)} =g$, where 
$\varepsilon$ is a sufficiently small positive real number. Hence, by the same argument 
as in the proof of Lemma \ref{x-lem11.1}, 
${}_\alpha L$ and ${}_{1-\alpha -\delta-\varepsilon}(L^\vee)$ are line bundles 
on $X(R)$ for some $0<R<1$. 
Moreover, $F$ is a generator of ${}_\alpha L$ and 
$G$ is a generator of ${}_{1-\alpha-\delta-\varepsilon} (L^\vee)$ on 
$X(R)$ for some $0<R<1$. 
We can easily check that 
\[
{}_\alpha (L|_{\pi^{-1}(Q)})=({}_\alpha L)|_{\pi^{-1}(Q)}
\]
for every $Q\in \Delta(R)^{n-1}$. 
By the same proof of Lemma \ref{x-lem11.3}, 
we see that $\gamma \left({}_\alpha (L|_{\pi^{-1}(Q)})\right)$ is independent 
of $Q\in \Delta(R)^{n-1}$. 
We finish the proof of Proposition \ref{x-prop11.4}. 
\end{proof}

The following theorem is 
the main result of this section. 

\begin{thm}\label{x-thm11.5}
Let $(L, h)$ be an acceptable line bundle on a partially 
punctured polydisk 
$\Delta^*\times \Delta^{n-1}$. 
Then ${}_\alpha L$ is a line 
bundle on $\Delta^n$ for any $\alpha \in \mathbb R$. 
Moreover, $\left({}_\alpha L\right)|_{\pi^{-1}(Q)}={}_{\alpha} 
\left(L|_{\pi^{-1}(Q)}\right)$ holds for every $Q\in \Delta^{n-1}$. 
We also have that $\gamma \left({}_{\alpha} 
\left(L|_{\pi^{-1}(Q)}\right)\right)$ is 
independent of $Q\in \Delta^{n-1}$. 
\end{thm}

\begin{proof}[Proof of Theorem \ref{x-thm11.5}]
We take an arbitrary point $P\in \Delta^{n-1}$. 
After shifting and rescaling the coordinate system around $P$, 
we apply Proposition \ref{x-prop11.4}. 
Then we have the desired properties. 
\end{proof}

\section{Acceptable vector bundles 
on $\Delta^*\times \Delta^{n-1}$}\label{x-sec12}

In this section, we study an acceptable vector bundle $(E,h)$ with
$\rank E\geq2$ on a partially punctured polydisk 
\[
X^*=\Delta^*\times \Delta^{n-1}.
\]
Our approach heavily relies on the results established in
Section~\ref{x-sec11}.

\begin{thm}\label{x-thm12.1}
Let $(E,h)$ be an acceptable vector bundle on a 
partially punctured polydisk $\Delta^*\times \Delta^{n-1}$.
Then, for any $\alpha\in\mathbb R$, ${}_\alpha E$ is locally free on
$\Delta^n$.
Moreover,
\[
{}_\alpha(E|_{\pi^{-1}(Q)})=({}_\alpha E)|_{\pi^{-1}(Q)}
\]
for every $Q\in \Delta^{n-1}$.
We also note that
\[
\Par_\alpha(E|_{\pi^{-1}(Q)},h|_{\pi^{-1}(Q)})
\]
is independent of $Q\in \Delta^{n-1}$.
\end{thm}

A more detailed description of the 
sheaf ${}_\alpha E$ and its local frames
can be found in the proof of Theorem~\ref{x-thm12.1}.

\begin{proof}[Proof of Theorem~\ref{x-thm12.1}]
We divide the proof into several steps.

\setcounter{step}{0}
\begin{step}\label{x-step12.1.1}
Set $P:=(0,\ldots,0)\in \Delta^{n-1}$.
In this step, we prove that ${}_\alpha E$ is locally free on $X(R)$ for some
$0<R<1$, under the assumption that
\[
\alpha \notin
\Par_\alpha(E|_{\pi^{-1}(P)},h|_{\pi^{-1}(P)}).
\]

Let $\{v_1,\ldots,v_r\}$ be a frame of
${}_\alpha(E|_{\pi^{-1}(P)})$. 
We note that ${}_\alpha (E|_{\pi^{-1}(P)})$ is locally 
free by Theorem \ref{x-thm9.1}. 
Fix $0<R<1$.
As in the line bundle case (cf.\ the proof of Lemma~\ref{x-lem11.1}),
by Proposition~\ref{x-prop8.5} and Corollary~\ref{x-cor7.5}, there exist
holomorphic sections $\{V_1,\ldots,V_r\}$ of ${}_\alpha E$ on $X^*(R)$ such
that
\[
V_i|_{\pi^{-1}(P)}=v_i \qquad (1\le i\le r).
\]

Consider the dual frame
\[
\{w_1,\ldots,w_r\}:=\{v_1^\vee,\ldots,v_r^\vee\}
\]
of ${}_{1-\alpha-\varepsilon}(E^\vee|_{\pi^{-1}(P)})$ for some
$0<\varepsilon\ll1$ (see \cite[Theorem~1.12]{fujino-fujisawa-ono}).
By the same argument, we obtain holomorphic sections
$\{W_1,\ldots,W_r\}$ of ${}_{1-\alpha-\varepsilon}(E^\vee)$ on $X^*(R)$ such
that
\[
W_i|_{\pi^{-1}(P)}=w_i=v_i^\vee .
\]

\begin{claim}\label{x-claim12.2}
For some $0<R<1$, the families $\{V_1,\ldots,V_r\}$ and
$\{W_1,\ldots,W_r\}$ form frames of ${}_\alpha E$ and
${}_{1-\alpha-\varepsilon}(E^\vee)$ on $X(R)$, respectively.
\end{claim}

\begin{proof}[Proof of Claim \ref{x-claim12.2}]
Note that, by construction,
\[
V_i \cdot W_j \in {}_{1-\varepsilon}\left(\mathcal O_{X^*}\right)
= \mathcal O_X,
\qquad
(V_i \cdot W_j)\big|_{\pi^{-1}(P)} = \delta_{ij}.
\]
Let
\[
A := (V_i \cdot W_j)_{i,j}
\]
be the associated $r\times r$ matrix-valued holomorphic function.
Shrinking the radius if necessary, we may assume that
\[
\det A \neq 0 \quad \text{on } X(R)
\]
for some $0<R<1$.
Define
\[
\begin{pmatrix}
W'_1 \\ \vdots \\ W'_r
\end{pmatrix}
=
A^{-1}
\begin{pmatrix}
W_1 \\ \vdots \\ W_r
\end{pmatrix}.
\]
Replacing $W_i$ by $W'_i$, we may further assume that
\[
V_i \cdot W_j = \delta_{ij}
\quad \text{on } X(R).
\]
We also assume that
\[
V_1 \wedge \cdots \wedge V_r \neq 0
\quad \text{on } X^*(R).
\]

Let $\Phi \in {}_\alpha E$.
Then $\Phi$ admits the expansion
\[
\Phi = \sum_{i=1}^r (\Phi \cdot W_i)\, V_i,
\]
where
\[
\Phi \cdot W_i \in {}_{1-\varepsilon}\left(\mathcal O_{X^*}\right)
= \mathcal O_X.
\]
It follows that $\{V_1,\ldots,V_r\}$ forms a local frame of
${}_\alpha E$ on $X(R)$.

Similarly, one checks that $\{W_1,\ldots,W_r\}$ is a local frame of
${}_{1-\alpha-\varepsilon}(E^\vee)$.
In particular, we obtain the identification
\[
{}_{1-\alpha-\varepsilon}(E^\vee) = ({}_\alpha E)^\vee.
\] 
We complete the proof. 
\end{proof}

By the same argument, for any $Q\in \Delta(R)^{n-1}$,
the restrictions $\{V_i|_{\pi^{-1}(Q)}\}$ and $\{W_i|_{\pi^{-1}(Q)}\}$ form
frames of ${}_\alpha(E|_{\pi^{-1}(Q)})$ and
${}_{1-\alpha-\varepsilon}(E^\vee|_{\pi^{-1}(Q)})$, respectively. 

The wedge product $V_1\wedge \cdots \wedge V_r$ defines a frame of
$\det({}_\alpha E)$, and similarly
$W_1\wedge \cdots \wedge W_r$ defines the dual frame of
$\det({}_{1-\alpha-\varepsilon}(E^\vee))$ on $X(R)$.
Moreover, for any $Q\in \Delta(R)^{n-1}$,
the restrictions
\[
(V_1\wedge \cdots \wedge V_r)|_{\pi^{-1}(Q)}
\quad \text{and} \quad
(W_1\wedge \cdots \wedge W_r)|_{\pi^{-1}(Q)}
\]
provide a frame of
$\det({}_\alpha(E|_{\pi^{-1}(Q)}))$ and the dual frame of
$\det({}_{1-\alpha-\varepsilon}(E^\vee|_{\pi^{-1}(Q)}))$,
respectively.
By the same argument as in the proof of Lemma~\ref{x-lem11.3}, it follows that
\[
\gamma\!\left({}_\alpha(E|_{\pi^{-1}(Q)})\right)
\]
is independent of $Q\in \Delta(R)^{n-1}$. 
\end{step}

\begin{step}\label{x-step12.1.2}
In this step, we prove that
$\Par_\alpha(E|_{\pi^{-1}(Q)},h|_{\pi^{-1}(Q)})$ is independent of
$Q\in \Delta^{n-1}$.

Set
\[
\Lambda:=\{(z_2,\ldots,z_n)\in \Delta^{n-1}\mid
\re z_i,\im z_i\in\mathbb Q \text{ for } 2\le i\le n\},
\]
and define
\[
\mathcal P:=\{\alpha\in\mathbb R\mid
\alpha\in \Par_\alpha(E|_{\pi^{-1}(Q)},h|_{\pi^{-1}(Q)})
\text{ for some } Q\in\Lambda\}.
\]

Let $P\in \Lambda$ be any point.
After shifting and rescaling coordinates around $P$, we may apply
Step~\ref{x-step12.1.1} to any $\alpha\in\mathbb R\setminus\mathcal P$.
It follows that, for such $\alpha$,
$\gamma({}_\alpha(E|_{\pi^{-1}(Q)}))$ is independent of
$Q\in \Delta^{n-1}$.
Since $\mathcal P$ is countable, Lemma~\ref{x-lem9.2} (i) 
implies that this
holds for all $\alpha\in\mathbb R$.
By Lemma~\ref{x-lem9.2} (ii), we conclude that
$\Par_\alpha(E|_{\pi^{-1}(Q)},h|_{\pi^{-1}(Q)})$ is independent of
$Q\in \Delta^{n-1}$.
\end{step}

\begin{step}\label{x-step12.1.3}
In this step, we prove that for any $\alpha\in \mathbb R$,
${}_\alpha E$ is locally free on $X(R)$ for some $0<R<1$.

Set $P:=(0,\ldots,0)\in \Delta^{n-1}$.
By Step~\ref{x-step12.1.1}, we may assume that
\[
\alpha \in
\Par_\alpha(E|_{\pi^{-1}(P)},h|_{\pi^{-1}(P)}).
\]
Let $\{v_1,\ldots,v_r\}$ be a frame of
${}_\alpha(E|_{\pi^{-1}(P)})$.
As in Step~\ref{x-step12.1.1}, there exist holomorphic sections
$\{V_1,\ldots,V_r\}$ of ${}_{\alpha+\delta}E$ on $X^*(R)$ for some
$0<\delta\ll1$ such that
\[
V_i|_{\pi^{-1}(P)}=v_i \qquad (1\le i\le r).
\]
Since $0<\delta\ll 1$, by Step~\ref{x-step12.1.2}, 
\[
\alpha+\delta\not\in \Par_{\alpha+\delta}\left(E|_{\pi^{-1}(P)}, 
h|_{\pi^{-1}(P)}\right)=\Par_{\alpha+\delta}\left(E|_{\pi^{-1}(Q)}, 
h|_{\pi^{-1}(Q)}\right)
\] 
holds for any $Q\in \Delta^{n-1}$. 
Thus, by Proposition~\ref{x-prop7.4}, we see that
$V_i\in {}_\alpha E$ for every $i$ since 
$0<\delta\ll 1$.

Let
\[
\{w_1,\ldots,w_r\}:=\{v_1^\vee,\ldots,v_r^\vee\}
\]
be the dual frame of
${}_{1-\alpha-\delta-\varepsilon}(E^\vee|_{\pi^{-1}(P)})$ for some
$0<\varepsilon\ll1$ (see \cite[Theorem~1.12]{fujino-fujisawa-ono}).
Arguing as in Step~\ref{x-step12.1.1}, we obtain holomorphic sections
$\{W_1,\ldots,W_r\}$ of ${}_{1-\alpha-\delta-\varepsilon}(E^\vee)$ on
$X^*(R)$ such that
\[
W_i|_{\pi^{-1}(P)}=w_i=v_i^\vee \qquad (1\le i\le r).
\]
By the argument in the proof of Claim~\ref{x-claim12.2},
$\{V_1,\ldots,V_r\}$ forms a frame of ${}_\alpha E$ on $X(R)$ for some
$0<R<1$.
In particular, ${}_\alpha E$ is locally free on $X(R)$.
\end{step}

\begin{step}\label{x-step12.1.4}
In this final step, we prove that ${}_\alpha E$ is locally free on
$\Delta^n$ for any $\alpha\in \mathbb R$.

Let $P\in \Delta^{n-1}$ be an arbitrary point.
After shifting and rescaling coordinates around $P$, we apply
Step~\ref{x-step12.1.3}.
It follows that ${}_\alpha E$ is locally free on $\Delta^n$ for every
$\alpha\in \mathbb R$.
\end{step}

We complete the proof of Theorem~\ref{x-thm12.1}.
\end{proof}

\section{Acceptable bundles on $(\Delta^*)^l\times \Delta^{n-l}$}
\label{x-sec13}

In this section, we prove Theorem~\ref{x-thm1.1}, 
which is one of the main
results of this paper, in full generality.

\medskip

Let $(E,h)$ be an acceptable vector bundle on a partially 
punctured polydisk 
\[
X^*:=(\Delta^*)^l\times \Delta^{n-l},
\]
where $l\geq2$.
As before, we set
\[
X:=\Delta^n
=\{(z_1,\ldots,z_n)\in\mathbb C^n \mid |z_i|<1 \text{ for all } i\},
\]
and
\[
X^*
=\left\{(z_1,\ldots,z_n)\in\mathbb C^n \ \middle|\
\begin{aligned}
&0<|z_i|<1 \quad \text{for } 1\le i\le l,\\
&|z_i|<1 \quad \text{for } l+1\le i\le n
\end{aligned}
\right\}.
\]
For each $1\le i\le n$, let
\[
\pi_i\colon X^*\to D_i:=\{z_i=0\}
\]
denote the natural projection.
For $1\le i\le l$, we also set
\[
D_i^\circ:=D_i\setminus \bigcup_{\substack{j\ne i\\ j\le l}} D_j.
\] 
We put $I:=\{1, \ldots, l\}$ and define 
\[
D_I:=\bigcap _{i\in I} D_i. 
\]

\medskip

We begin with the following basic lemma.

\begin{lem}\label{x-lem13.1}
Let $(E,h)$ be an acceptable vector bundle on a partially 
punctured polydisk 
$X^*=(\Delta^*)^l\times \Delta^{n-l}$.
For any $\bm a=(a_1,\ldots,a_l)\in\mathbb R^l$, define
\[
\Par(_{\bm a}E,i)
:=\Par\!\left(_{a_i}\!\left(E|_{\pi_i^{-1}(P)}\right)\right)
\]
for a point $P\in D_i^\circ$.
Then $\Par(_{\bm a}E,i)$ is independent of the choice of
$P\in D_i^\circ$, and hence is well defined.
\end{lem}

\begin{proof}[Proof of Lemma \ref{x-lem13.1}]
This follows immediately from Theorem~\ref{x-thm12.1} and 
its proof.
\end{proof}

We first consider a special case.

\begin{prop}\label{x-prop13.2} 
Let $(E,h)$ be an acceptable vector bundle on a 
partially punctured polydisk 
$X^*=(\Delta^*)^l\times \Delta^{n-l}$.
Assume that
\begin{equation}\label{x-eq13.1}
\Par({}^\diamond\!E,i)\subset \left(-\frac{1}{l},\,0\right]
\end{equation}
for every $1\le i\le l$.
Then ${}^\diamond\!E$ is locally free on $X(R)$ for some $0<R<1$.
\end{prop}

The proof of Proposition~\ref{x-prop13.2} closely follows the arguments
in Sections~\ref{x-sec11} and~\ref{x-sec12}.

\begin{proof}
By Theorem~\ref{x-thm12.1}, we may assume that $l\ge2$.
Set
\[
Y^*
:=\{(z_1,\ldots,z_n)\in X^* \mid
z_1=\cdots=z_l,\; z_{l+1}=\cdots=z_n=0\}.
\]
Let $\{v_1,\ldots,v_r\}$ be a frame of ${}^\diamond\!(E|_{Y^*})$.
Choose a sufficiently large positive real number $N$ such that
$h(0,N)$ is Nakano semipositive.
As in Proposition~\ref{x-prop8.5}, define
\[
\psi:=\frac{1}{l}\sum_{i=1}^l \log|z_i|^2,\qquad
\phi:=-\left(1-\frac{1}{l}\right)\sum_{i=1}^l \log|z_i|^2,
\]
and
\[
\phi_{\bm\varepsilon_1/l}
:=-\frac{\varepsilon_1}{l}\sum_{i=1}^l \log|z_i|^2,
\]
where $\varepsilon_1>0$ is sufficiently small.
Let $0<R<1$ and put $z:=z_1=\cdots=z_l$ on $Y^*$.
Then
\[
\int_{Y^*(R)} |v_j|_{h(0,N)}
e^{-\psi-\phi_{\bm\varepsilon_1/l}}
\frac{\sqrt{-1}}{2}\,dz\wedge d\overline z
<\infty
\]
for every $j$.
Hence, by Proposition~\ref{x-prop8.5}, there exist holomorphic sections
$\{V_1,\ldots,V_r\}$ of $E$ on $X^*(R)$ such that
$V_j|_{Y^*}=v_j$ and
\[
\int_{X^*(R)} |V_j|_{h(0,N)}
e^{-\phi-\phi_{\bm\varepsilon_1/l}}
\frac{\omega_P^n}{n!}
<\infty
\]
for all $j$.
By Corollary~\ref{x-cor7.5}, we obtain
\[
V_j\in {}_{\left(1-\frac{1}{l}+\frac{\varepsilon_1}{l},\ldots,
1-\frac{1}{l}+\frac{\varepsilon_1}{l}\right)}E
\]
for every $j$. 

Let $\{w_1,\ldots,w_r\} := \{v_1^\vee,\ldots,v_r^\vee\}$ be 
the dual frame of
${}_{1-\varepsilon}(E^\vee|_{Y^*})$ for some $0<\varepsilon\ll1$. 
We may assume that
$h^\vee(0,N) := h^\vee \cdot e^{-\chi(0,N)}$ is Nakano semipositive since $N$ is sufficiently large.
We put
\[
\phi_{(\bm 1-\bm \varepsilon_2)/l} := -\frac{1-\varepsilon_2}{l}
\sum_{i=1}^l
\log |z_i|^2,
\]
where $0<\varepsilon_2<\varepsilon$.
Then
\[
\int_{Y^*(R)} |w_j|_{h^\vee(0,N)}\,
e^{-\psi-\phi_{(\bm 1-\bm \varepsilon_2)/l}}
\frac{\sqrt{-1}}{2}\,dz\wedge d\overline z < \infty.
\]
Applying Proposition~\ref{x-prop8.5} and Corollary~\ref{x-cor7.5} 
once again,
we obtain holomorphic sections $\{W_1,\ldots,W_r\}$ of $E^\vee$ on $X^*(R)$
such that $W_j|_{Y^*}=w_j=v_j^\vee$ and
\[
W_j \in {}_{\left(1-\frac{1}{l}+\frac{1-\varepsilon_2}{l},\ldots,
1-\frac{1}{l}+\frac{1-\varepsilon_2}{l}\right)}(E^\vee)
\]
for every $j$.

By~\eqref{x-eq13.1} and Proposition~\ref{x-prop7.4}, we have
$V_j\in {}^\diamond\!E$ for all $j$, since $0<\varepsilon_1\ll1$.
Moreover,
\[
W_j\in {}_{\left(1-\frac{\varepsilon_2}{l},\ldots,
1-\frac{\varepsilon_2}{l}\right)}(E^\vee)
\]
for every $j$.
Arguing as in the proof of Claim~\ref{x-claim12.2}, we conclude that
$\{V_1,\ldots,V_r\}$ forms a frame of ${}^\diamond\!E$ on $X(R)$ and that
$\{W_1,\ldots,W_r\}$ is the dual frame, for some $0<R<1$.
\end{proof}

As an immediate consequence of Proposition~\ref{x-prop13.2}, we obtain
the following corollary.

\begin{cor}\label{x-cor13.3} 
Let $(E,h)$ be an acceptable vector bundle on a partially punctured
polydisk $X^*=(\Delta^*)^l\times \Delta^{n-l}$. 
Let $\bm a=(a_1, \ldots, a_l)\in \mathbb R^l$. 
Assume that
\[
\Par(_{\bm a}E,i)\subset
\left(-\frac{1}{l}+a_i,\,a_i\right]
\]
for every $1\le i\le l$.
Then ${}_{\bm a}E$ is locally free on $X$.
\end{cor}

\begin{proof}
By Lemma~\ref{x-lem6.8}, we may assume that $a_i=0$ for all $i$.
Let $P\in X\setminus X^*$ be an arbitrary point.
After shrinking and rescaling coordinates around $P$, we apply
Proposition~\ref{x-prop13.2}.
It follows that ${}_{\bm a}E$ is locally free on $\Delta^n$ for every
$\bm a\in\mathbb R^l$.
\end{proof}

From now on, we study acceptable vector bundles on 
$X^*=(\Delta^*)^l \times \Delta^{n-l}$ in general. 
Lemma \ref{x-lem13.4} follows easily from Corollary \ref{x-cor13.3}.

\begin{lem}\label{x-lem13.4}
Let $(E, h)$ be an acceptable vector bundle 
on $X^*=(\Delta^*)^l\times \Delta^{n-l}$, 
and let $\eta>0$ be sufficiently small. 
Then there exists a positive integer $c$ such that
\[
\Par\bigl({}_{\bm \eta} (\psi_c^{-1}E), i\bigr) \subset (-\eta, \eta)
\] 
for each $i=1, \dots, l$, 
where $\bm \eta=(\eta, \dots, \eta)\in \mathbb R^l$ and 
$\psi_c \colon X=\Delta^n\to X=\Delta^n$ is the finite cover
defined by
\[
\psi_c(z_1, \dots, z_l, z_{l+1}, \dots, z_n) =
(z_1^c, \dots, z_l^c, z_{l+1}, \dots, z_n).
\] 
More precisely, for any positive integer $m$, we can choose $c$ divisible by $m$. 
Moreover, if $0<\eta<\frac{1}{2l}$, then ${}_{\bm \eta} (\psi_c^{-1}E)$ 
is locally free on $X$.
\end{lem}

\begin{proof}
Note first that $\psi_c^*\omega_P = \omega_P$. 
Hence, $\psi_c^*(E, h)$ is an acceptable vector bundle on $X^*$. 
From now on, we simply write $\psi^{-1}_cE$ to denote $\psi^*_c(E, h)$.

To analyze the parabolic weights of $\psi_c^{-1}E$ along $\{z_i=0\}$, 
it suffices to consider the case $l=1$. 
In this case, the behavior of parabolic weights under $\psi_c$ 
follows from the curve case (see Lemma \ref{x-lem10.2} and 
\cite[Section 11]{fujino-fujisawa-ono}). 

Applying Diophantine approximation 
(cf. \cite[Lemma 12.2]{fujino-fujisawa-ono} 
and \cite[Chapter I, Theorem VI]{cas}), 
we may choose a sufficiently large and divisible positive integer $c$ such that
\[
\Par\bigl({}_{\bm \eta} (\psi_c^{-1}E), i\bigr) \subset (-\eta, \eta)
\]
for all $i=1, \dots, l$. 
Moreover, \cite[Lemma 12.2]{fujino-fujisawa-ono} ensures that $c$ can be taken 
to be divisible by any given positive integer $m$. 

Finally, if $0<\eta<\frac{1}{2l}$, Corollary \ref{x-cor13.3} implies 
that ${}_{\bm \eta} (\psi_c^{-1}E)$ is locally free on $X$.
\end{proof}

The following theorem is the main result of this section. 
Although we use Lemma \ref{x-lem13.4}, which differs slightly from 
\cite[Lemma 21.7.2]{mochizuki4}, the proof of Theorem \ref{x-thm13.5} 
is essentially the same as that given 
in \cite[21.7.2.~Proof of Theorem 21.3.1]{mochizuki4}.

\begin{thm}[Prolongation by increasing orders]\label{x-thm13.5}
Let $(E, h)$ be an acceptable vector bundle on 
a partially punctured polydisk 
$X^*=(\Delta^*)^l\times \Delta^{n-l}$. 
Then, for any $\bm a\in \mathbb R^l$, ${}_{\bm a} E$ is 
a locally free sheaf on $X=\Delta^n$.
\end{thm}

\begin{proof}[Proof of Theorem \ref{x-thm13.5}] 
We first note that, by Lemma~\ref{x-lem6.8}, we may assume without loss of
generality that $\bm a=\bm 0$.

We divide the proof into two steps.
In Step~\ref{x-step13.5.1}, we prove that
${}_{\bm a}E = {}^\diamond\!E$ is locally free on $X$.
In Step~\ref{x-step13.5.2}, we give a supplementary remark on the
parabolic filtrations; the description obtained there will be used in
the proof of Theorem~\ref{x-thm1.3}.

\setcounter{step}{0}
\begin{step}\label{x-step13.5.1}

The case $l=1$ has already been treated in Section~\ref{x-sec11}.
Hence, we assume $l\ge 2$ throughout this step.

Let $0<\eta<\frac{1}{2l}$ and set
$\bm \eta=(\eta,\dots,\eta)\in \mathbb R^l$.
We consider $X=\Delta^n$ and $X^*=(\Delta^*)^l\times \Delta^{n-l}$.
For a positive integer $c$, define
\[
\psi_c\colon X\to X, \quad
\psi_c(z_1,\dots,z_n)=(z_1^c,\dots,z_l^c,z_{l+1},\dots,z_n).
\]
We choose $c$ so that
\[
\Par\!\left({}_{\bm \eta} (\psi_c^{-1} E), i\right)
\subset (-\eta,\eta), \quad i=1,\dots,l.
\]
By Lemma~\ref{x-lem13.4}, the sheaf
${}_{\bm \eta} (\psi_c^{-1} E)$ is locally free.

Let $\mu_c=\mathbb Z/c\mathbb Z=\langle g\rangle$.
There is a natural $\mu_c^l$-action on $X$ given by
\[
(g_1,\dots,g_l)^*(z_1,\dots,z_n)
=(\zeta_1 z_1,\dots,\zeta_l z_l, z_{l+1},\dots,z_n),
\]
where $g_i$ is a generator of the $i$-th factor
$\mu_c^{(i)}$ of $\mu_c^l$ and $\zeta_i$ is a primitive $c$-th root of
unity.
This action lifts to ${}_{\bm \eta} (\psi_c^{-1} E)$, and each
$\mu_c^{(i)}$ acts on
${}_{\bm \eta} (\psi_c^{-1} E)|_{D_i}$.

We have a vector bundle decomposition
\[
{}_{\bm \eta} (\psi_c^{-1} E)|_{D_i}
= \bigoplus_{0\le p\le c-1} {}^i V_p,
\]
where $g_i$ acts on ${}^i V_p$ by multiplication by $\zeta_i^{-p}$.
As in the curve case (see \eqref{x-eq10.3}), we define a map
\[
\varphi_i\colon
\{\,0\le p\le c-1 \mid {}^i V_p\neq 0\,\}
\longrightarrow
\Par({}_{\bm \eta} (\psi_c^{-1} E), i).
\]

For $\eta-1<b\le \eta$, we define a filtration
${}^i\!F'$ of ${}_{\bm \eta} (\psi_c^{-1} E)|_{D_i}$ in the category of
vector bundles on $D_i$ by
\begin{equation}\label{x-eq13.2}
{}^i\!F'_b
:= \bigoplus_{\varphi_i(p)\le b} {}^i V_p .
\end{equation}
The collection of filtrations
$\bigl({}^i\!F'\mid i=1,\ldots,l\bigr)$
is compatible in the sense of Definition~\ref{x-def3.4},
since $\mu_c^l$ is abelian.
In particular, we obtain a vector bundle decomposition
\begin{equation}\label{x-eq13.3}
{}_{\bm \eta}\!\left(\psi_c^{-1}E\right)\big|_{D_I}
= \bigoplus_{\bm p} {}^I V_{\bm p},
\end{equation}
where $\bm p=(p_1,\ldots,p_l)\in\{0,1,\ldots,c-1\}^l$,
and $g_i$ acts on ${}^I V_{\bm p}$ by multiplication by
$\zeta_i^{-p_i}$ for each $1\le i\le l$.

We set
\[
\bm{\delta}_i :=
(\underbrace{0,\ldots,0}_{i-1},1,0,\ldots,0)\in\mathbb R^l .
\]
For $-1<b<0$, we define a subsheaf
${}_{\bm \eta+b\bm \delta_i} (\psi_c^{-1} E)'$
of ${}_{\bm \eta} (\psi_c^{-1} E)$ by
\[
{}_{\bm \eta+b\bm \delta_i} (\psi_c^{-1} E)'
:= \Ker\!\left(
\pi\colon
{}_{\bm \eta} (\psi_c^{-1} E)\longrightarrow
\frac{{}_{\bm \eta} (\psi_c^{-1} E)|_{D_i}}{{}^i\!F'_{\eta+b}}
\right),
\]
where $\pi$ is the natural morphism of $\mathcal O_X$-modules.

\begin{claim}\label{x-claim13.6}
For any $-1<b<0$, we have
\[
{}_{\bm \eta+b\bm \delta_i} (\psi_c^{-1} E)'
=
{}_{\bm \eta+b\bm \delta_i} (\psi_c^{-1} E).
\]
In particular, the parabolic filtration ${}^i\!F$ coincides with
${}^i\!F'$.
\end{claim}

\begin{proof}[Proof of Claim \ref{x-claim13.6}]
Let $f\in {}_{\bm \eta+b\bm \delta_i} (\psi_c^{-1} E)$.
Viewing $f$ as a section of ${}_{\bm \eta} (\psi_c^{-1} E)$, we set
$\overline f := \pi(f)$.
For any point $P\in D_i^\circ$, we have
$f|_{\pi_i^{-1}(P)} \in
{}_{\bm \eta}(\psi_c^{-1} E|_{\pi_i^{-1}(P)})$.
By the curve case,
\[
f(P)=f|_{\pi_i^{-1}(P)}(P)\in {}^i\!F'_{\eta+b}|_P,
\]
and hence $\overline f(P)=0$.
Since this holds for all $P\in D_i^\circ$, we obtain
$\overline f=0$ on $D_i$, which shows
$f\in {}_{\bm \eta+b\bm \delta_i} (\psi_c^{-1} E)'$.

Conversely, let
$f\in {}_{\bm \eta+b\bm \delta_i} (\psi_c^{-1} E)'$.
For any $P\in D_i^\circ$, the curve case implies
\[
\left|f|_{\pi_i^{-1}(P)}\right|_h
= O\!\left(\frac{1}{|z_i|^{\eta+b+\varepsilon}}\right)
\] 
for all $\varepsilon >0$. 
By Proposition~\ref{x-prop7.4}, this shows that
$f\in {}_{\bm \eta+b\bm \delta_i} (\psi_c^{-1} E)$.

Hence, we have the desired equality 
\[
{}_{\bm \eta+b\bm \delta_i} (\psi_c^{-1} E)' =
{}_{\bm \eta+b\bm \delta_i} (\psi_c^{-1} E),
\] 
which completes the proof of Claim \ref{x-claim13.6}.
\end{proof}

We record the following elementary observation.

\begin{claim}\label{x-claim13.7}
Let $v \in {}^I V_{\bm p}$, and let $v^\sharp$ be a holomorphic section of
${}_{\bm \eta}(\psi_c^{-1}E)$ on $X(R)$ for some $0<R<1$ such that
$v^\sharp|_{D_I}=v$.
Define
\[
v^\flat :=
\frac{1}{c^l}
\sum_{k_1=0}^{c-1}\cdots\sum_{k_l=0}^{c-1}
\zeta_1^{k_1 p_1}\cdots \zeta_l^{k_l p_l}
(g_1^{k_1},\ldots,g_l^{k_l})^* v^\sharp .
\]
Then $v^\flat|_{D_I}=v$, and $v^\flat$ is a $\mu_c^l$-equivariant
holomorphic section of ${}_{\bm \eta}(\psi_c^{-1}E)$ on $X(R)$.
\end{claim}

We now return to the proof of Theorem~\ref{x-thm13.5}.
By \eqref{x-eq13.3} and Claim~\ref{x-claim13.7}, we can choose a
$\mu_c^l$-equivariant frame
$\bm v=\{v_1,\ldots,v_r\}$ of
${}_{\bm \eta} (\psi_c^{-1} E)$
on $X(R)$ for some $0<R<1$ such that
\[
(g_1,\ldots,g_l)^* v_i
= \prod_{j=1}^l \zeta_j^{-p_j(v_i)}\, v_i
\]
for integers $0\le p_j(v_i)\le c-1$.
By construction, the frame $\bm v$ is compatible with the parabolic
filtrations ${}^j\!F$ for all $1\le j\le l$
(see Definition~\ref{x-def3.5}).

For each $i$, define
\[
\overline v_i := \prod_{j=1}^l z_j^{p_j(v_i)} \cdot v_i .
\]
Since $\overline v_i$ is $\mu_c^l$-invariant, it descends to a section of
$E$.
By the curve case (Lemma~\ref{x-lem10.3}), each $\overline v_i$ is a
section of ${}^\diamond\!E$.
Moreover, for any $P\in D_i^\circ$, the restrictions
$\overline{\bm v}|_{\pi_i^{-1}(P)}$ form a frame of
${}^\diamond\!(E|_{\pi_i^{-1}(P)})$.
It follows that $\overline{\bm v}$ is a frame of ${}^\diamond\!E$ on a
neighborhood of the origin.
Hence, ${}^\diamond\!E$ is locally free on $X(R)$ for some $0<R<1$.
As in Step~\ref{x-step12.1.4} of the proof of
Theorem~\ref{x-thm12.1}, we conclude that ${}^\diamond\!E$ is locally
free on $X$.
\end{step}

\begin{step}\label{x-step13.5.2}
In this step, we give a more direct description of the
parabolic filtrations of ${}^\diamond\!E$ in a neighborhood of the
origin.
This description will be used in the proof of
Theorem~\ref{x-thm1.3}.

As in the curve case (see \eqref{x-eq10.2} in
Section~\ref{x-sec10}), we have a map
\[
\chi_i \colon
\{\,0 \le p \le c-1 \mid {}^i V_p \neq 0\,\}
\longrightarrow
\Par({}^\diamond\!E, i).
\]
We set
\begin{equation}\label{x-eq13.4}
a_i(v_j) := \chi_i\bigl(p_i(v_j)\bigr).
\end{equation}

We define a filtration ${}^i\!F'_b$ of ${}^\diamond\!E|_{D_i}$ by vector
subbundles by
\[
{}^i\!F'_b :=
\bigl\langle \overline v_j|_{D_i} \,\big|\, a_i(v_j) \le b \bigr\rangle,
\]
that is, ${}^i\!F'_b$ is the vector subbundle of
${}^\diamond\!E|_{D_i}$ generated by those
$\overline v_j|_{D_i}$ with $a_i(v_j)\le b$.
Here
$\overline{\bm v}=\{\overline v_1,\ldots,\overline v_r\}$
denotes the local frame of ${}^\diamond\!E$ constructed in
Step~\ref{x-step13.5.1}.

For $-1<b\le 0$, we define a subsheaf
\[
{}_{b\cdot \bm \delta_i}(E)' :=
\Ker\!\left(
\pi \colon
{}^\diamond\!E \longrightarrow
\frac{{}^\diamond\!E|_{D_i}}{{}^i\!F'_b}
\right),
\]
where $\pi$ denotes the natural morphism of $\mathcal O_X$-modules.

\begin{claim}\label{x-claim13.8}
We have
\[
{}_{b\cdot \bm \delta_i} E
=
{}_{b\cdot \bm \delta_i}(E)',
\]
and consequently ${}^i\!F_b = {}^i\!F'_b$.
\end{claim}

\begin{proof}[Proof of Claim \ref{x-claim13.8}]
Let $f \in {}_{b\cdot \bm \delta_i} E$.
We regard $f$ as a section of ${}^\diamond\!E$.
For any $P\in D_i^\circ$, applying the curve case to
\[
f|_{\pi_i^{-1}(P)} \in
{}^\diamond\!(E|_{\pi_i^{-1}(P)}),
\]
we obtain
$f(P)\in {}^i\!F'_b|_P$.
Hence, $f\in {}_{b\cdot \bm \delta_i}(E)'$.

Conversely, let $f\in {}_{b\cdot \bm \delta_i}(E)'$.
By the curve case, we have
\[
f|_{\pi_i^{-1}(P)} \in {}_b(E|_{\pi_i^{-1}(P)})
\quad \text{for all } P\in D_i^\circ.
\]
Therefore, by Proposition~\ref{x-prop7.4},
we conclude that $f\in {}_{b\cdot \bm \delta_i} E$. 

Thus, we obtain 
\[
{}_{b\cdot \bm \delta_i} E = {}_{b\cdot \bm \delta_i}(E)',
\] 
and consequently ${}^i\!F_b = {}^i\!F'_b$. 
\end{proof}

By construction, ${}^i\!F'$ defines a filtration in the category of
vector bundles on $D_i$, and the tuple
$\bigl({}^i\!F' \mid i=1,\ldots,l\bigr)$
is compatible in the sense of Definition~\ref{x-def3.4}.
Hence, the same holds for ${}^i\!F$:
it defines a filtration by vector subbundles on $D_i$, and the tuple
$\bigl({}^i\!F \mid i=1,\ldots,l\bigr)$ is compatible.
\end{step}

We conclude the proof of Theorem~\ref{x-thm13.5}.
\end{proof}

We prove Theorems~\ref{x-thm1.1} and~\ref{x-thm1.3}.

\begin{proof}[Proof of Theorem~\ref{x-thm1.3}]
As in the proof of Theorem~\ref{x-thm13.5}, we may assume that
$\bm a=\bm 0$, i.e.,
${}_{\bm a}E={}^\diamond\!E$.
Let
$\overline{\bm v}=\{\overline v_1,\ldots,\overline v_r\}$
be a local frame of ${}^\diamond\!E$ on a sufficiently small
open neighborhood $U$ of the origin in $\Delta^n$,
constructed in Step~\ref{x-step13.5.1} of the proof of
Theorem~\ref{x-thm13.5}.
Then we have
\[
{}^\diamond\!E|_U
=\bigoplus_{j=1}^r \mathcal O_U \cdot \overline v_j.
\]

For $1\le j\le r$ and $1\le i\le l$, we set
\[
a_i(\overline v_j):=a_i(v_j)\in(-1,0],
\]
as in \eqref{x-eq13.4} of Step~\ref{x-step13.5.2} in the proof of
Theorem~\ref{x-thm13.5}.
By the curve case result and Proposition \ref{x-prop7.4}, 
it follows that for any
$\bm b\in\mathbb R^l$,
\[
{}_{\bm b}E|_U
=\bigoplus_{j=1}^r
\mathcal O_U\!\left(
\sum_{i=1}^l \lfloor b_i-a_i(\overline v_j)\rfloor D_i
\right)\cdot \overline v_j.
\]
This completes the proof.
\end{proof}

\begin{proof}[Proof of Theorem~\ref{x-thm1.1}]
In Theorem~\ref{x-thm13.5}, we have already shown that
${}_{\bm a}E$ is a locally free sheaf on $\Delta^n$
for any $\bm a\in\mathbb R^l$.
By Theorem~\ref{x-thm1.3}, the family
\[
\bigl({}_{\bm a}E \mid \bm a\in\mathbb R^l\bigr)
\]
naturally forms a filtered bundle in the sense of Mochizuki
(see Definitions~\ref{x-def4.1} and \ref{x-def4.2}).
This completes the proof.
\end{proof}

\section{Weak norm estimates}\label{x-sec14}

In this section, we prove the weak norm estimate stated in
Theorem~\ref{x-thm1.4}.

\begin{proof}[Proof of Theorem~\ref{x-thm1.4}]
Let $\bm v=\{v_1,\ldots,v_r\}$ be a frame of ${}_{\bm a}E$
defined in a neighborhood of the origin $0\in\Delta^n$,
which is compatible with the parabolic filtrations
\[
\mathbf F:=\bigl({}^i\!F \mid i=1,\ldots,l\bigr).
\]
See Definition~\ref{x-def3.6} for details.

For $1\le i\le l$ and $1\le j\le r$, we set
\[
a_i(v_j):={}^i\!\deg^{\mathbf F}(v_j)
=\deg^{{}^i\!F}(v_j).
\]
Define
\[
v'_j:=v_j\cdot\prod_{i=1}^l |z_i|^{a_i(v_j)},
\qquad
\bm v':=\{v'_1,\ldots,v'_r\}.
\]
By the construction of $\bm v'$ and Proposition~\ref{x-prop7.4},
there exist constants $C_1>0$ and $M_1>0$ such that
\[
H(h,\bm v')
\le
C_1\Bigl(-\sum_{i=1}^l\log|z_i|\Bigr)^{M_1} I_r.
\]

Let $\bm v^\vee=\{v_1^\vee,\ldots,v_r^\vee\}$ be the dual frame of $\bm v$.
For any point $P\in D_i^\circ$,
the restriction $\bm v^\vee|_{\pi_i^{-1}(P)}$ is a frame of
\[
{}_{-a_i+(1-\varepsilon)}E^\vee\big|_{\pi_i^{-1}(P)}
\]
for $0<\varepsilon\ll1$, compatible with the induced parabolic filtration.
Hence, for $0<\varepsilon\ll1$,
$\bm v^\vee$ defines a local frame of
\[
{}_{-\bm a+(1-\varepsilon)\bm\delta}E^\vee
\]
around the origin $0\in\Delta^n$,
compatible with the parabolic filtrations,
where
\[
\bm\delta=(1,\ldots,1)\in\mathbb R^l.
\]

By the curve case, we have
\[
{}^i\!\deg^{\mathbf F}(v_j^\vee)
=\deg^{{}^i\!F}(v_j^\vee)
=-a_i(v_j)
\]
for all $i$ and $j$.
We define
\[
(v_j^\vee)':=
v_j^\vee\cdot\prod_{i=1}^l |z_i|^{-a_i(v_j)},
\qquad
(\bm v^\vee)':=\{(v_1^\vee)',\ldots,(v_r^\vee)'\}.
\]
Applying Proposition~\ref{x-prop7.4} again,
there exist constants $C_2>0$ and $M_2>0$ such that
\[
H\bigl(h^\vee,(\bm v^\vee)'\bigr)
\le
C_2\Bigl(-\sum_{i=1}^l\log|z_i|\Bigr)^{M_2} I_r.
\]
This implies that there exist constants $C_3>0$ and $M_3>0$ such that
\[
C_3\Bigl(-\sum_{i=1}^l\log|z_i|\Bigr)^{-M_3} I_r
\le
H(h,\bm v').
\]
Combining the above estimates, we obtain the desired weak norm estimate.
\end{proof}

\section{Basic properties via reduction to curves}\label{x-sec15}

In this final section, we establish Theorems~\ref{x-thm1.5},
\ref{x-thm1.6}, and~\ref{x-thm1.7}, together with
Corollary~\ref{x-cor1.8}, by systematically reducing the statements
to the curve case.

\begin{proof}[Proof of Theorem \ref{x-thm1.5}]
We first note the inclusion
\[
{}_{-\bm a + \bm 1 - \bm \varepsilon}\left(E^\vee\right)
\subset \left({}_{\bm a} E\right)^\vee.
\]
This follows directly from the definition. 

To prove the reverse inclusion, let
$\bm v = \{v_1,\ldots,v_r\}$ be a local frame of ${}_{\bm a}E$
compatible with the parabolic filtration, and let
$\bm v^\vee = \{v_1^\vee,\ldots,v_r^\vee\}$ denote the dual frame.
For any $P \in D_i^\circ$, we consider the restriction
$v_j^\vee|_{\pi_i^{-1}(P)}$.

By the curve case results
\cite[Theorems~1.12 and~13.2]{fujino-fujisawa-ono}
together with Proposition~\ref{x-prop7.4},
we conclude that
\[
v_j^\vee \in
{}_{-\bm a + \bm 1 - \bm \varepsilon}\left(E^\vee\right)
\quad \text{for all } j.
\]
This implies
\[
\left({}_{\bm a}E\right)^\vee
\subset
{}_{-\bm a + \bm 1 - \bm \varepsilon}\left(E^\vee\right).
\]
Combining the two inclusions, we obtain the desired equality
\[
\left({}_{\bm a}E\right)^\vee
=
{}_{-\bm a + \bm 1 - \bm \varepsilon}\left(E^\vee\right).
\]
We complete the proof of Theorem \ref{x-thm1.5}. 
\end{proof}

\begin{proof}[Proof of Theorem \ref{x-thm1.6}]
By definition, we have the inclusion
\begin{equation}\label{x-eq15.1}
\sum_{\bm a_1 + \bm a_2 \leq \bm b}
{}_{\bm a_1}E_1 \otimes {}_{\bm a_2}E_2
\subset
{}_{\bm b}(E_1 \otimes E_2).
\end{equation}
Thus, it suffices to show that this inclusion is in fact an equality.

Set
\[
Y := \{ z_1 = \cdots = z_l,\; z_{l+1} = \cdots = z_n = 0 \}
\subset \Delta^n.
\]
We consider the restriction
${}_{\bm b}(E_1 \otimes E_2)|_Y$.
Applying the curve case result
\cite[Theorem~1.14]{fujino-fujisawa-ono}
to this restriction, we obtain that
the inclusion \eqref{x-eq15.1} is an equality
in a neighborhood of the origin.

The same argument applies after translating the center
to any point of
$\Delta^n \setminus (\Delta^*)^l \times \Delta^{n-l}$.
Hence, the inclusion \eqref{x-eq15.1} is an equality everywhere,
which completes the proof of Theorem~\ref{x-thm1.6}.
\end{proof}

\begin{proof}[Proof of Theorem \ref{x-thm1.7}]
Recall that
\[
\Hom(E_1, E_2) = E_1^\vee \otimes E_2
\]
is an acceptable vector bundle (see Lemma \ref{x-lem6.2}).
By definition, a section
$f \in {}_{\bm a}\!\Hom(E_1, E_2)$
satisfies the condition that
\[
f({}_{\bm k}E_1) \subset {}_{\bm a + \bm k}E_2
\quad \text{for all } \bm k \in \mathbb{R}^l.
\]

Conversely, let
$f \in \Hom(E_1, E_2)$
be a morphism satisfying
\[
f({}_{\bm k}E_1) \subset {}_{\bm a + \bm k}E_2
\quad \text{for all } \bm k \in \mathbb{R}^l.
\]
For any $P \in D_i^\circ$, we consider the restriction
$f|_{\pi_i^{-1}(P)}$.
By the curve case result
\cite[Proposition~17.1]{fujino-fujisawa-ono}
together with Proposition~\ref{x-prop7.4},
we conclude that
\[
f \in {}_{\bm a}\!\Hom(E_1, E_2).
\] 
This completes the proof. 
\end{proof}

Corollary~\ref{x-cor1.8} follows 
directly from Theorem~\ref{x-thm1.7}.

\begin{proof}[Proof of Corollary \ref{x-cor1.8}]
The assertion follows immedizately from Theorem~\ref{x-thm1.7},
since
\[
{}^\diamond\!\End(E) = {}_{\bm 0}\!\Hom(E, E).
\] 
This completes the proof. 
\end{proof}


\end{document}